\documentclass[11pt,onecolumn,letterpaper,twoside]{IEEEtran}

\usepackage{color}
\usepackage{subcaption}

\usepackage{graphicx}				% Use pdf, png, jpg, or eps§ with pdflatex; use eps in DVI mode
\usepackage{amssymb,amsmath,amsthm}
\usepackage{enumitem}

\usepackage{algorithm}

\newtheorem{Def}{Definition}

\newtheorem{Assumption}{Assumption}

\newtheorem{Thm}{Theorem}
\newtheorem{Fact}{Fact}
\newtheorem{Lem}{Lemma}
\newtheorem{Cor}{Corollary}

\DeclareMathOperator*{\argmin}{argmin}
\newcommand*{\tran}{^{\mkern-1.5mu\mathsf{T}}}

\title{Online Convex Optimization with Stochastic Constraints}

\author{Hao Yu, Michael J. Neely, Xiaohan Wei\\Department of Electrical Engineering \\ 
University of Southern California}

\begin{document}
\maketitle

\begin{abstract}
This paper considers online convex optimization  (OCO) with stochastic constraints, which generalizes Zinkevich's OCO over a  known simple fixed set by introducing multiple stochastic functional constraints that are i.i.d. generated at each round and are disclosed to the decision maker only after the decision is made. This formulation arises naturally when decisions are restricted by stochastic environments or deterministic environments with noisy observations. It also includes many important problems  as special cases, such as OCO with long term constraints, stochastic constrained convex optimization, and deterministic constrained convex optimization.  To solve this problem, this paper proposes a new algorithm that achieves $O(\sqrt{T})$ expected regret and constraint violations and $O(\sqrt{T}\log(T))$ high probability regret and constraint violations. Experiments on a real-world data center scheduling problem further verify the performance of the new algorithm.
\end{abstract}

%\begin{keywords}
%online convex optimization, long term constraints, regret bounds, constraint violation bounds, low complexity
%\end{keywords}

\section{Introduction}

Online convex optimization (OCO) is a multi-round learning process with arbitrarily-varying convex loss functions where the decision maker has to choose decision $x(t)\in \mathcal{X}$ before observing the corresponding loss function $f^t(\cdot)$. For a fixed time horizon $T$,  define the \emph{regret} of a learning algorithm with respect to the best fixed decision in hindsight (with full knowledge of all loss functions) as
\begin{align*}
\text{regret}(T) = \sum_{t=1}^{T} f^t(\mathbf{x}(t))  - \min_{\mathbf{x}\in \mathcal{X}} \sum_{t=1}^T f^t(\mathbf{x}).
\end{align*}
The goal of OCO is to develop dynamic learning algorithms such that regret grows sub-linearly with respect to $T$.  The setting of OCO is introduced in a series of work \cite{Cesa-Bianchi96TNN, Kivinen97,Gordon99COLT,Zinkevich03ICML} and is formalized in \cite{Zinkevich03ICML}. OCO has gained considerable amount of research interest recently with various applications such as online regression, prediction with expert advice, online ranking, online shortest paths and portfolio selection.  See  \cite{Shalev-Shwartz11FoundationTrends,Hazan16FoundationTrends} for more applications and backgrounds. 

In \cite{Zinkevich03ICML}, Zinkevich shows that using an online gradient descent (OGD) update given by 
\begin{align}
\mathbf{x}(t+1) = \mathcal{P}_{\mathcal{X}} \big[ \mathbf{x}(t) - \gamma \nabla f^t (\mathbf{x}(t))\big] \label{eq:Zinkevich}
\end{align}
where $ \nabla f^t (\cdot)$ is a subgradient of $f^t(\cdot)$ and $\mathcal{P}_{\mathcal{X}}[\cdot]$ is the projection onto set $\mathcal{X}$ can achieve $O(\sqrt{T})$ regret. Hazan et al. in \cite{Hazan07ML} show that better regret is possible under the assumption that each loss function is strongly convex but $O(\sqrt{T})$ is the best possible if no additional assumption is imposed.

It is obvious that Zinkevich's OGD in \eqref{eq:Zinkevich} requires the full knowledge of set $\mathcal{X}$ and low complexity of the projection $\mathcal{P}_\mathcal{X}[\cdot]$.  However, in practice, the constraint set $\mathcal{X}$, which is often described by many functional inequality constraints, can be time varying and may not be fully disclosed to the decision maker.  In \cite{Mannor09JMLR}, Mannor et al. extend OCO by considering time-varying constraint functions $g^t(\mathbf{x})$ which can arbitrarily vary and are only disclosed to us after each $\mathbf{x}(t)$ is chosen.  In this setting, Mannor et al. in \cite{Mannor09JMLR} explore the possibility of designing learning algorithms such that regret grows sub-linearly and $\limsup_{T\rightarrow\infty} \frac{1}{T}\sum_{t=1}^T g^t(\mathbf{x}(t)) \leq 0$, i.e., the (cumulative) constraint violation $\sum_{t=1}^T g^t(\mathbf{x}(t))$ also grows sub-linearly.   Unfortunately, Mannor et al. in \cite{Mannor09JMLR} prove that this is impossible even when both $f^t(\cdot)$ and $g^t(\cdot)$ are simple linear functions.

Given the impossibility results shown by Mannor et al. in \cite{Mannor09JMLR},  this paper considers OCO where constraint functions $g^t(\mathbf{x})$ are not arbitrarily varying but independently and identically distributed (i.i.d.) generated from an unknown probability model. More specifically, this paper considers \emph{online convex optimization (OCO) with stochastic constraint} $\mathcal{X}=\{\mathbf{x}\in \mathcal{X}_0:  \mathbb{E}_\omega[g_k(\mathbf{x}; \omega)] \leq 0, k\in\{1,2,\ldots,m\}\}$ where $\mathcal{X}_0$ is a known fixed set; the expressions of stochastic constraints $\mathbb{E}_\omega[g_k(\mathbf{x}; \omega)]$ (involving expectations with respect to $\omega$ from an unknown distribution) are unknown; and subscripts $k\in\{1,2,\ldots, m\}$ indicate the possibility of multiple functional constraints.  In OCO with stochastic constraints, the decision maker receives loss function $f^t(\mathbf{x})$ and i.i.d. constraint function realizations  $g_k^t (\mathbf{x}) \overset{\Delta}{=}g_k(\mathbf{x}; \omega(t))$ at each round $t$. However, the expressions of $g_k^t (\cdot)$ and $f^t(\cdot)$ are disclosed to the decision maker only after decision $\mathbf{x}(t)\in \mathcal{X}_0$ is chosen. This setting arises naturally when decisions are restricted by stochastic environments or deterministic environments with noisy observations. For example, if we consider online routing (with link capacity constraints) in wireless networks \cite{Mannor09JMLR}, each link capacity is not a fixed constant (as in wireline networks) but an i.i.d. random variable since wireless channels are stochastically time-varying by nature \cite{book_FundamentalWireless}. OCO with stochastic constraints also covers important special cases such as OCO with long term constraints \cite{Mahdavi12JMLR,Cotter15COLT, Jenatton16ICML}, stochastic constrained convex optimization \cite{Mahdavi13NIPS} and deterministic constrained convex optimization \cite{book_ConvexOpt_Nesterov}.

Let $\mathbf{x}^{\ast} = \argmin_{\{\mathbf{x}\in\mathcal{X}_0:  \mathbb{E}[g_k(\mathbf{x}; \omega)]\leq 0, \forall k\in \{1,2,\ldots,m\}\}} \sum_{t=1}^{T}f^{t}(\mathbf{x})$ be the best fixed decision in hindsight (knowing all loss functions $f^{t}(\mathbf{x})$ and the distribution of stochastic constraint functions $g_k(\mathbf{x}; \omega)$). Thus, $\mathbf{x}^\ast$ minimizes the $T$-round cumulative loss and satisfies all stochastic constraints in expectation, which also implies $\limsup_{T\rightarrow\infty} \frac{1}{T}\sum_{t=1}^T g_k^t(\mathbf{x}^\ast) \leq 0$ almost surely by the strong law of large numbers.  Our goal is to develop dynamic learning algorithms that guarantee both regret $\sum_{t=1}^{T} f^t(\mathbf{x}(t))  - \sum_{t=1}^T f^t(\mathbf{x}^\ast)$ and constraint violations $\sum_{t=1}^T g_k^t(\mathbf{x}(t))$ grow sub-linearly.  

Note that Zinkevich's algorithm in \eqref{eq:Zinkevich} is not applicable to OCO with stochastic constraints since $\mathcal{X}$ is unknown and it can happen that $\mathcal{X}(t)=\{\mathbf{x}\in \mathcal{X}_0: g_k (\mathbf{x}; \omega (t))\leq 0, \forall k\in\{1,2,\ldots, m\}\} = \emptyset$ for certain realizations $\omega(t)$, such that projections $\mathcal{P}_{\mathcal{X}}[\cdot]$ or $\mathcal{P}_{\mathcal{X}(t)}[\cdot]$ required in \eqref{eq:Zinkevich}  are not even well-defined.

{\bf Our Contributions:}  This paper solves online convex optimization with stochastic constraints. In particular, we propose a new learning algorithm that is proven to achieve $O(\sqrt{T})$ expected regret and constraint violations and $O(\sqrt{T}\log(T))$ high probability regret and constraint violations.  Along the way, we developed new techniques for stochastic analysis, e.g., Lemma \ref{lm:drift-random-process-bound}, and improve upon state-of-the-art results in the following special cases.

\begin{itemize}[leftmargin=20pt]
\item \emph{OCO with long term constraints:} This is a special case where each $g_k^t(\mathbf{x})\equiv g_k (\mathbf{x})$ is known and does not depend on time. Note that $\mathcal{X} = \{\mathbf{x}\in \mathcal{X}_0: g_k(\mathbf{x})\leq 0, \forall k\in\{1,2,\ldots,m\}\}$ can be complicated while $\mathcal{X}_0$ might be a simple hypercube. To avoid high complexity involved in the projection onto $\mathcal{X}$ as in Zinkevich's algorithm, work in \cite{Mahdavi12JMLR,Cotter15COLT, Jenatton16ICML} develops low complexity algorithms that use projections onto a simpler set $\mathcal{X}_0$ by allowing  $g_k(\mathbf{x}(t))>0$ for certain rounds but ensuring $\limsup_{T\rightarrow \infty} \frac{1}{T}\sum_{t=1}^T g_k(\mathbf{x}(t))\leq 0$. The best existing performance is $O(T^{\max\{\beta, 1-\beta\}})$ regret and $O(T^{1-\beta/2})$ constraint violations where $\beta\in (0,1)$ is an algorithm parameter \cite{Jenatton16ICML}. This gives $O(\sqrt{T})$ regret with worse $O(T^{3/4})$ constraint violations or $O(\sqrt{T})$ constraint violations with worse $O(T)$ regret.  In contrast, our algorithm, which only uses projections onto $\mathcal{X}_0$ as shown in Lemma \ref{lm:x-update}, can achieve $O(\sqrt{T})$ regret and $O(\sqrt{T})$ constraint violations simultaneously.\footnote{By adapting the methodology presented in this paper, our other report \cite{YuNeely16ArxivOnlineOpt} developed a different algorithm that can only solve the special case problem ``OCO with long term constraints" but can achieve $O(\sqrt{T})$ regret and $O(1)$ constraint violations. The current paper also relaxes a deterministic Slater condition assumption required in our other technical report \cite{NeelyYu17ArXiv} for OCO with time-varying constraints, which requires the existence of constant $\epsilon>0$ and fixed point $\hat{\mathbf{x}}\in \mathcal{X}_0$ such that $g_k(\hat{\mathbf{x}};\omega(t))\leq-\epsilon, \forall k\in\{1,2,\ldots,m\}$ for all $\omega(t)\in \Omega$. By relaxing the deterministic Slater condition assumption to the stochastic Slater condition in Assumption \ref{as:interior-point}, the current paper even allows the possibility that $g_k(\mathbf{x};\omega(t))\leq 0$ is infeasible for certain $\omega(t)\in \Omega$. However, under the deterministic Slater condition assumption, our technical report \cite{NeelyYu17ArXiv} shows that if the regret is defined as the cumulative loss difference between our algorithm and the best fixed point from set $\mathcal{A} = \{\mathbf{x}\in\mathcal{X}: g_k(\mathbf{x}; \omega) \leq 0, \forall k\in\{1,2,\ldots,m\}, \forall \omega \in \Omega\}$, which is called a common subset in \cite{NeelyYu17ArXiv}, then our algorithm can achieve $O(\sqrt{T})$ regret and $O(\sqrt{T})$ constraint violations simultaneously even if the constraint functions $g_k^t(\mathbf{x})$ are arbitrarily time-varying (not necessarily i.i.d.). That is, by imposing the additional deterministic Slater condition and restricting the regret to be defined over the common subset $\mathcal{A}$, our algorithm can escape the impossibility shown by Mannor et al. in \cite{Mannor09JMLR}. To the best of our knowledge, this is the first time that specific conditions are proposed to enable sublinear regret and constraints violations simultaneously for OCO with arbitrarily time-varying constraint functions.  Since the current paper focuses on  OCO with stochastic constraints, we refer interested readers to Section IV in \cite{NeelyYu17ArXiv} for results on OCO with arbitrarily time-varying constraints.}
\item \emph{Stochastic constrained convex optimization:} This is a special case where each $f^{t}(\mathbf{x})$ is i.i.d. generated from an unknown distribution. This problem has many applications in operation research and machine learning such as Neyman-Pearson classification and risk-mean portfolio. The work \cite{Mahdavi13NIPS} develops a (batch) offline algorithm that produces a solution with high probability performance guarantees only after sampling the problems for sufficiently many times. That is, during the process of sampling, there are no performance guarantees. The work \cite{Lan16arXiv} proposes a stochastic approximation based (batch) offline algorithm for stochastic convex optimization with one single stochastic functional inequality constraint. In contrast, our algorithm is an online algorithm with online performance guarantees. \footnote{While the analysis of this paper assumes a Slater-type condition, note that our other work \cite{NeelyYu17ArXiv} shows that the Slater condition is not needed in the special case when both the objective and constraint functions vary  i.i.d. over time. (This also includes the case of deterministic constrained convex optimization, since processes that do not vary with time are indeed i.i.d. processes.) In such scenarios, Section VI in our work \cite{NeelyYu17ArXiv} shows that our algorithm works more generally whenever a Lagrange multiplier vector attaining the strong duality exists.} 

\item \emph{Deterministic constrained convex optimization:} This is a special case where each $f^t(\mathbf{x})\equiv f (\mathbf{x})$ and $g_k^t(\mathbf{x})\equiv g_k (\mathbf{x})$ are known and do not depend on time.  In this case, the goal is to develop a fast algorithm that converges to a good solution (with a small error) with a few number of iterations; and our algorithm with $O(\sqrt{T})$ regret and constraint violations is equivalent to an iterative numerical algorithm with $O(1/\sqrt{T})$ convergence rate.  Our algorithm is subgradient based and does not require the smoothness or differentiability of the convex program.  Recall that Nesterov in \cite{book_ConvexOpt_Nesterov} shows that $O(1/\sqrt{T})$ is the best possible convergence rate of any subgradient/gradient based algorithm for non-smooth convex programs. Thus, our algorithm is optimal. The primal-dual subgradient method considered in \cite{Nedic09_PrimalDualSubgradient} has the same $O(1/\sqrt{T})$ convergence rate but requires an upper bound of optimal Lagrange multipliers, which is typically unknown in practice.
Our algorithm does not require such bounds to be known.  
\end{itemize}

\section{Formulation and New Algorithm}

Let $\mathcal{X}_{0}$ be a known fixed compact convex set.  Let $g_k(\mathbf{x}; \omega(t)), k\in \{1,2,\ldots, m\}$ be sequences of functions that are i.i.d. realizations of stochastic constraint functions $\tilde{g}_{k}(\mathbf{x}) \overset{\Delta}{=} \mathbb{E}_{\omega}[g_{k}(\mathbf{x};\omega)]$ with random variable $\omega\in \Omega$ from an unknown distribution. That is, $\omega(t)$ are i.i.d. samples of $\omega$.  Let $f^t(\mathbf{x})$ be a sequence of convex functions that can arbitrarily vary as long as each $f^t(\cdot)$ is independent of all $\omega(\tau)$ with $\tau \geq t + 1$ so that we are unable to predict future constraint functions based on the knowledge of the current loss function. For example, each $f^t(\cdot)$ can even be chosen adversarially based on $\omega(\tau), \tau\in\{1,2,\ldots, t\}$ and actions $\mathbf{x}(\tau), \tau\in\{1,2,\ldots,t\}$. For each $\omega\in \Omega$, we assume $g_k(\mathbf{x}; \omega)$ are convex with respect to $\mathbf{x}\in \mathcal{X}_{0}$.  At the beginning of each round $t$, neither the loss function $f^{t}(\mathbf{x})$ nor the constraint function realizations $g_{k}(\mathbf{x}; \omega(t))$ are known to the decision maker. However, the decision maker still needs to make a decision $\mathbf{x}(t)\in \mathcal{X}_0$ for round $t$; and after that  $f^{t}(\mathbf{x})$ and $g_{k}(\mathbf{x}, \omega(t))$ are disclosed to the decision maker at the end of round $t$. 

For convenience, we often suppress the dependence of each $g_{k}(\mathbf{x}; \omega(t))$ on $\omega(t)$ and write $g_{k}^{t}(\mathbf{x}) = g_{k}(\mathbf{x}; \omega(t))$. Recall $\tilde{g}_{k}(\mathbf{x}) = \mathbb{E}_{\omega}[g_{k}(\mathbf{x};\omega)]$ where the expectation is with respect to $\omega$. Define $\mathcal{X}  =\{\mathbf{x}\in \mathcal{X}_0:  \tilde{g}_k(\mathbf{x}) = \mathbb{E}[g_k(\mathbf{x}; \omega)]\leq 0, \forall k\in \{1,2,\ldots,m\}\}$.  We further define the stacked vector of multiple functions $g^t_{1}(\mathbf{x}), \ldots, g^t_{m}(\mathbf{x})$ as $\mathbf{g}^t(\mathbf{x})= [g^{t}_{1}(\mathbf{x}), \ldots, g^{t}_{m}(\mathbf{x})]\tran$ and define $\tilde{\mathbf{g}}(\mathbf{x}) = [\mathbb{E}_{\omega}[g_{1}(\mathbf{x}; \omega)], \ldots, \mathbb{E}_{\omega}[g_{m}(\mathbf{x}; \omega)]]\tran$.  We use $\Vert \cdot\Vert$ to denote the Euclidean norm for a vector. Throughout this paper, we have the following assumptions:
\begin{Assumption}[Basic Assumptions] \label{as:basic}~
\begin{itemize}[leftmargin=20pt]
\item Loss functions $f^t(\mathbf{x})$ and constraint functions $g_k(\mathbf{x}; \omega)$ have bounded subgradients on $\mathcal{X}_{0}$. That is, there exists $D_{1}>0 $  and $D_{2}>0$ such that $\Vert \nabla f^{t} (\mathbf{x})\Vert \leq D_{1}$ for all $\mathbf{x}\in \mathcal{X}_{0}$ and all $t\in \{0,1,\ldots\}$ and $\Vert \nabla g_{k}(\mathbf{x}; \omega) \Vert \leq D_{2}$ for all $\mathbf{x} \in \mathcal{X}_{0}$, all $\omega\in \Omega$ and all $k\in\{1,2,\ldots,m\}$.\footnote{We use $\nabla h(\mathbf{x})$ to denote a subgradient of a convex function 
$h$ at the point $\mathbf{x}$. 
If the gradient exists, then $\nabla h(\mathbf{x})$ is the gradient.  Nothing in this paper requires gradients to exist: We only need the basic subgradient inequality   $h(\mathbf{y}) \geq h(\mathbf{x}) + [\nabla h(\mathbf{x})]\tran [\mathbf{y}-\mathbf{x}]$ for all $\mathbf{x}, \mathbf{y} \in \mathcal{X}_0$.}
\item There exists constant $G>0$ such that $\Vert \mathbf{g}(\mathbf{x}; \omega)\Vert \leq G$ for all $\mathbf{x}\in \mathcal{X}_0$ and all $\omega\in \Omega$.  
\item There exists constant $R>0$ such that $\Vert \mathbf{x} - \mathbf{y}\Vert\leq R$ for all $\mathbf{x}, \mathbf{y}\in \mathcal{X}_{0}$.
\end{itemize}
\end{Assumption}

\begin{Assumption}[The Slater Condition]\label{as:interior-point}
There exists $\epsilon>0$ and $\hat{\mathbf{x}}\in \mathcal{X}_{0}$ such that $\tilde{g}_{k}(\hat{\mathbf{x}}) = \mathbb{E}_{\omega} [g_{k}(\hat{\mathbf{x}}; \omega)]\leq -\epsilon$ for all $k\in\{1,2,\ldots, m\}$.
\end{Assumption}

\subsection{New Algorithm}

Now consider the following algorithm described in Algorithm \ref{alg:new-alg}. This algorithm chooses $\mathbf{x}(t+1)$ as the decision for round $t+1$ based on $f^{t} (\cdot)$ and $\mathbf{g}^{t}(\cdot)$ without requiring $f^{t+1}(\cdot)$ or $\mathbf{g}^{t+1}(\cdot)$.  

\begin{algorithm} 
\caption{}
\label{alg:new-alg}
Let $V, \alpha$ be constant algorithm parameters.  Choose $\mathbf{x}(1) \in \mathcal{X}_{0}$ arbitrarily and let $Q_{k}(1) = 0, \forall k\in\{1,2,\ldots,m\}$.  At the end of each round $t\in\{1,2,\ldots\}$, observe $f^{t}(\cdot)$ and $\mathbf{g}^{t}(\cdot)$ and do the following:
\begin{itemize}[leftmargin=15pt]
\item  Choose $\mathbf{x}(t+1)$ that solves
\vskip-20pt
\begin{align}
 \min_{\mathbf{x}\in \mathcal{X}_{0}}\big\{ V [\nabla f^{t}(\mathbf{x}(t))]\tran [\mathbf{x} - \mathbf{x}(t)] +  \sum_{k=1}^{m} Q_{k}(t)[\nabla g^{t}_{k}(\mathbf{x}(t))]\tran [\mathbf{x} - \mathbf{x}(t)] + \alpha \Vert \mathbf{x} - \mathbf{x}(t)\Vert^{2} \big\}  \label{eq:x-update}
\end{align}
\vskip -10pt
as the decision for the next round $t+1$, where $\nabla f^{t}(\mathbf{x}(t))$ is a subgradient of $f^{t}(\mathbf{x})$ at point $\mathbf{x}=\mathbf{x}(t)$ and $\nabla g^{t}_{k}(\mathbf{x}(t))$ is a subgradient of $g^{t}_{k}(\mathbf{x})$ at point $\mathbf{x} = \mathbf{x}(t)$.
\item Update each virtual queue $Q_k(t+1), \forall k\in\{1,2,\ldots,m\}$ via 
\vskip -17pt
\begin{align}
Q_{k}(t+1) = \max\left\{Q_{k}(t) + g^{t}_{k}(\mathbf{x}(t)) + [\nabla g_{k}^{t} (\mathbf{x}(t))]\tran [\mathbf{x}(t+1) - \mathbf{x}(t)], 0\right\},  \label{eq:queue-update}
\end{align}
\vskip -5pt
where $\max\{\cdot, \cdot\}$ takes the larger one between two elements.
\end{itemize}
\end{algorithm}

For each stochastic constraint function $g_k(\mathbf{x}; \omega)$, we introduce $Q_k(t)$ and call it a virtual queue since its dynamic is similar to a queue dynamic. The next lemma summarizes that $\mathbf{x}(t+1)$ update in \eqref{eq:x-update} can be implemented via a simple projection onto $\mathcal{X}_0$. 

\begin{Lem} \label{lm:x-update}
The $\mathbf{x}(t+1)$ update  in \eqref{eq:x-update} is given by $\mathbf{x}(t+1) = \mathcal{P}_{\mathcal{X}_{0}}\big[\mathbf{x}(t) - \frac{1}{2\alpha} \mathbf{d}(t)\big]$,
where $\mathbf{d}(t) = V\nabla f^{t}(\mathbf{x}(t)) + \sum_{k=1}^{m} Q_{k}(t) \nabla g^{t}_{k}(\mathbf{x}(t))$ and $\mathcal{P}_{\mathcal{X}_{0}}[\cdot]$ is the projection onto convex set $\mathcal{X}_{0}$.
\end{Lem}

\begin{proof}
The projection by definition is $\min_{\mathbf{x}\in \mathcal{X}_0}\Vert \mathbf{x} - [\mathbf{x}(t) - \frac{1}{2\alpha} \mathbf{d}(t)]\Vert^2$ and is equivalent to \eqref{eq:x-update}.
\end{proof}

\subsection{Intuitions of Algorithm \ref{alg:new-alg} }
Note that if there are no stochastic constraints $g^t_{k}(\mathbf{x})$, i.e., $\mathcal{X} = \mathcal{X}_0$, then Algorithm \ref{alg:new-alg} has $Q_k(t) \equiv 0, \forall t$ and becomes Zinkevich's algorithm  with $\gamma = \frac{V}{2\alpha}$ in \eqref{eq:Zinkevich} since
\begin{align}
\mathbf{x}(t+1) \overset{(a)}{=}  \argmin_{\mathbf{x}\in \mathcal{X}_0} \big\{ \underbrace{V [\nabla f^{t}(\mathbf{x}(t))]\tran [\mathbf{x} - \mathbf{x}(t)] + \alpha \Vert \mathbf{x} - \mathbf{x}(t)\Vert^{2}}_{\text{penalty}} \big\} \overset{(b)}{=} \mathcal{P}_{\mathcal{X}_0} \big [\mathbf{x}(t) - \frac{V}{2\alpha} \nabla f^{t} (\mathbf{x}(t))\big] \label{eq:new-alg-as-Zinkevich}
\end{align}
where (a) follows from \eqref{eq:x-update}; and (b) follows from Lemma \ref{lm:x-update} by noting that $\mathbf{d}(t) = V\nabla f^t(\mathbf{x}(t))$. Call  the term marked by an underbrace in \eqref{eq:new-alg-as-Zinkevich} the \emph{penalty}. Thus, Zinkevich's algorithm is to minimize the \emph{penalty} term and is a special case of Algorithm \ref{alg:new-alg} used to solve OCO over $\mathcal{X}_0$.

Let $\mathbf{Q}(t) = \big[ Q_1(t), \ldots, Q_m(t)\big]\tran$ be the vector of virtual queue backlogs. Let $L(t) = \frac{1}{2} \Vert \mathbf{Q}(t)\Vert^2$ be a \emph{Lyapunov function} and define  \emph{Lyapunov drift}

\begin{align}
\Delta (t) = L(t+1) - L(t) = \frac{1}{2} [ \Vert \mathbf{Q}(t+1)\Vert^{2} - \Vert \mathbf{Q}(t)\Vert^{2}]. \label{eq:def-drift}
\end{align}

The intuition behind Algorithm \ref{alg:new-alg} is to choose $\mathbf{x}(t+1)$ to minimize an upper bound of the expression
\begin{align}
\underbrace{\Delta(t)}_{\text{drift}} + \underbrace{V [\nabla f^{t}(\mathbf{x}(t))]\tran [\mathbf{x} - \mathbf{x}(t)] + \alpha \Vert \mathbf{x} - \mathbf{x}(t)\Vert^{2}}_{\text{penalty}} \label{eq:def-dpp}
\end{align}
The intention to minimize penalty is natural since Zinkevich's algorithm (for OCO without stochastic constraints) minimizes penalty, while the intention to minimize drift is motivated by observing that $g_k^{t}(\mathbf{x}(t))$ is accumulated into queue $Q_k(t+1)$ introduced in \eqref{eq:queue-update} such that we intend to have small queue backlogs.  The drift $\Delta(t)$ can be complicated and is in general non-convex. The next lemma  provides a simple upper bound of $\Delta(t)$ and follows directly from \eqref{eq:queue-update}. 

\begin{Lem}\label{lm:drift}At each round $t\in\{1,2,\ldots\}$, Algorithm \ref{alg:new-alg} guarantees

\begin{align}
\Delta(t) \leq  \sum_{k=1}^{m} Q_{k}(t) \big[g_{k}^{t}(\mathbf{x}(t))  + [\nabla g_{k}^{t}(\mathbf{x}(t))]\tran [\mathbf{x}(t+1) - \mathbf{x}(t)]\big] +   \frac{1}{2}(G+\sqrt{m}D_{2} R)^{2},  \label{eq:drift-upper-bound}
\end{align}
where $m$ is the number of constraint functions; and $D_{2}, G$ and $R$ are defined in Assumption \ref{as:basic}.
\end{Lem}
\begin{proof}  
%See Appendix \ref{app:pf-lm-drift}. 
Recall that for any $b\in \mathbb{R}$, if $a = \max\{b, 0\}$ then $a^{2} \leq b^{2}$. Fix $k\in\{1,2,\ldots, m\}$. The virtual queue update equation  $Q_{k}(t+1) = \max\left\{Q_{k}(t) + g^{t}_{k}(\mathbf{x}(t)) + [\nabla g_{k}^{t} (\mathbf{x}(t))]\tran [\mathbf{x}(t+1) - \mathbf{x}(t)], 0\right\}$ implies that 
\begin{align}
\frac{1}{2} [Q_{k}(t+1)] ^{2} \leq& \frac{1}{2} \big[ Q_{k}(t) + g^{t}_{k}(\mathbf{x}(t)) + [\nabla g_{k}^{t} (\mathbf{x}(t))]\tran [\mathbf{x}(t+1) - \mathbf{x}(t)]\big]^{2} \nonumber \\
=&  \frac{1}{2} [Q_{k}(t)]^{2} + Q_{k}(t) \big[g_{k}^{t}(\mathbf{x}(t))  + [\nabla g_{k}^{t}(\mathbf{x}(t))]\tran [\mathbf{x}(t+1) - \mathbf{x}(t)]\big]  \nonumber\\& \quad+  \frac{1}{2} \big[g_{k}^{t}(\mathbf{x}(t))  + [\nabla g_{k}^{t}(\mathbf{x}(t))]\tran [\mathbf{x}(t+1) - \mathbf{x}(t)]\big]^{2}\nonumber\\
\overset{(a)}{=}& \frac{1}{2} [Q_{k}(t)]^{2} + Q_{k}(t) \big[g_{k}^{t}(\mathbf{x}(t))  + [\nabla g_{k}^{t}(\mathbf{x}(t))]\tran [\mathbf{x}(t+1) - \mathbf{x}(t)]\big]  + \frac{1}{2} [ h_{k}]^{2}, \label{eq:pf-lm-drift-eq1}
\end{align}
where (a) follows by defining $h_{k} = g_{k}^{t}(\mathbf{x}(t)) + [\nabla g_{k}^{t}(\mathbf{x}(t))]\tran [\mathbf{x}(t+1) - \mathbf{x}(t)]$.  

Define $\mathbf{s} = [s_{1},\ldots, s_{m}]\tran$, where $s_{k} = [\nabla g_{k}^{t}(\mathbf{x}(t))]\tran [\mathbf{x}(t+1) - \mathbf{x}(t)], \forall k\in\{1,2,\ldots, m\}$; and $\mathbf{h} = [h_{1}, \ldots, h_{m}]\tran = \mathbf{g}^{t}(\mathbf{x}(t)) + \mathbf{s}$. Then, 
\begin{align}
\Vert \mathbf{h}\Vert \overset{(a)}{\leq}& \Vert \mathbf{g}^{t}(\mathbf{x}(t))\Vert + \Vert \mathbf{s}\Vert 
\overset{(b)}{\leq}  G + \sqrt{\sum_{k=1}^{m} D^{2}_{2} R^{2}} 
= G+ \sqrt{m}D_{2} R, \label{eq:pf-lm-drift-eq2}
\end{align}
where (a) follows from the triangle inequality; and (b) follows from the definition of Euclidean norm, the Cauchy-Schwartz inequality and Assumption \ref{as:basic}.

Summing \eqref{eq:pf-lm-drift-eq1} over $k\in\{1,2,\ldots,m\}$ yields 
\begin{align*}
\frac{1}{2}\Vert\mathbf{Q}(t+1)\Vert^{2} \leq& \frac{1}{2} \Vert \mathbf{Q}(t)\Vert^{2} + \sum_{k=1}^{m} Q_{k}(t)\big[g_{k}^{t}(\mathbf{x}(t))  + [\nabla g_{k}^{t}(\mathbf{x}(t))]\tran [\mathbf{x}(t+1) - \mathbf{x}(t)]\big]  + \Vert \mathbf{h}\Vert^2 \\
\overset{(a)}{\leq} &\frac{1}{2} \Vert \mathbf{Q}(t)\Vert^{2} +  \sum_{k=1}^{m} Q_{k}(t)\big[g_{k}^{t}(\mathbf{x}(t))  + [\nabla g_{k}^{t}(\mathbf{x}(t))]\tran [\mathbf{x}(t+1) - \mathbf{x}(t)]\big] + \frac{1}{2}(G+ \sqrt{m} D_{2} R)^{2},
\end{align*}
where (b) follows from \eqref{eq:pf-lm-drift-eq2}. Rearranging the terms yields the desired result.
\end{proof}
At the end of round $t$, $\sum_{k=1}^m Q_k(t) g_k^t(\mathbf{x}(t)) + \frac{1}{2}[G+\sqrt{m}D_{2} R]^{2}$ is a given constant that is not affected by decision $\mathbf{x}(t+1)$. The algorithm decision in \eqref{eq:x-update} is now transparent: $\mathbf{x}(t+1)$ is chosen to minimize the drift-plus-penalty expression \eqref{eq:def-dpp}, where $\Delta(t)$ is approximated by the bound in \eqref{eq:drift-upper-bound}. 

\subsection{Preliminary Analysis and More Intuitions of Algorithm \ref{alg:new-alg}}

The next lemma  relates constraint violations and virtual queue values and follows directly from \eqref{eq:queue-update}.
\begin{Lem}\label{lm:queue-constraint-inequality}
For any $T\geq 1$, Algorithm \ref{alg:new-alg} guarantees $\sum_{t=1}^{T} g^{t}_{k}(\mathbf{x}(t)) \leq  \Vert \mathbf{Q}(T+1)\Vert  +  D_2 \sum_{t=1}^{T} \Vert \mathbf{x}(t+1) - \mathbf{x}(t)\Vert, \forall k\in\{1,2,\ldots,m\}$, where $D_{2}$ is defined in Assumption \ref{as:basic}.
\end{Lem}
\begin{proof}
%See Appendix \ref{app:pf-lm-queue-constraint-inequality}.
Fix $k\in\{1,2,\ldots,m\}$ and $T \geq 1$.  For any $t \in \{0,1, \ldots \}$, \eqref{eq:queue-update} in Algorithm \ref{alg:new-alg} gives: 
\begin{align*}
Q_k(t+1) &= \max\{Q_{k}(t) + g^{t}_{k}(\mathbf{x}(t)) + [\nabla g_{k}^{t} (\mathbf{x}(t))]\tran [\mathbf{x}(t+1) - \mathbf{x}(t)], 0\} \\
& \geq Q_{k}(t) + g^{t}_{k}(\mathbf{x}(t)) + [\nabla g_{k}^{t} (\mathbf{x}(t))]\tran [\mathbf{x}(t+1) - \mathbf{x}(t)]\\
&\overset{(a)}{\geq } Q_{k}(t) + g^{t}_{k}(\mathbf{x}(t)) - \Vert\nabla g_{k}^{t} (\mathbf{x}(t)) \Vert \Vert \mathbf{x}(t+1) - \mathbf{x}(t)\Vert\\
&\overset{(b)}{\geq } Q_{k}(t) + g^{t}_{k}(\mathbf{x}(t)) - D_{2} \Vert \mathbf{x}(t+1) - \mathbf{x}(t)\Vert, 
\end{align*}
where (a) follows from the Cauchy-Schwartz inequality and (b) follows from Assumption \ref{as:basic}. Rearranging terms yields
\begin{align*}
g_{k}^{t}(\mathbf{x}(t)) &\leq Q_k(t+1) - Q_k(t) + D_{2} \Vert \mathbf{x}(t+1) - \mathbf{x}(t)\Vert.
\end{align*}
Summing over $t \in \{1, \ldots, T\}$ yields 
\begin{align*}
\sum_{t=1}^{T} g_{k}^{t}(\mathbf{x}(t)) &\leq Q_{k}(T+1) - Q_{k}(1) + D_2 \sum_{t=1}^{T} \Vert \mathbf{x}(t+1)  - \mathbf{x}(t)\Vert\\
&\overset{(a)}{=} Q_{k}(T+1) + D_2 \sum_{t=1}^{T} \Vert \mathbf{x}(t+1)  - \mathbf{x}(t)\Vert \\
&\leq \Vert \mathbf{Q}(T+1)\Vert + D_2 \sum_{t=1}^{T} \Vert \mathbf{x}(t+1)  - \mathbf{x}(t)\Vert.
\end{align*} 
where (a) follows from the fact $Q_k(1) = 0$. 
\end{proof}
Recall that function $h:\mathcal{X}_0\rightarrow\mathbb{R}$ is said to be 
\emph{$c$-strongly convex} if $h(\mathbf{x}) - \frac{c}{2}\Vert \mathbf{x}\Vert^2$ is convex over $\mathbf{x}\in\mathcal{X}_0$. By the definition of strongly convex functions, it is easy to see that if  
$\phi:\mathcal{X}_0\rightarrow\mathbb{R}$ is a convex function, then for any constant $c>0$ and any constant vector $\mathbf{x}_0$, the function $\phi(x) + \frac{c}{2}\Vert\mathbf{x}-\mathbf{x}_0\Vert^2$ is $c$-strongly convex.  Further, it is known that if $h:\mathcal{X}\rightarrow\mathbb{R}$ is a $c$-strongly convex function and is minimized at point $\mathbf{x}^{min} \in \mathcal{X}_0$, then (see, for example, Corollary 1 in \cite{YuNeely17SIOPT}): 
\begin{align} 
h(\mathbf{x}^{min}) \leq h(\mathbf{x}) - \frac{c}{2}\Vert \mathbf{x}-\mathbf{x}^{min}\Vert^2 \quad \forall \mathbf{x} \in \mathcal{X}_0 \label{eq:strongly} 
\end{align} 
Note that the expression involved in minimization \eqref{eq:x-update} in Algorithm \ref{alg:new-alg} is strongly convex with modulus $2\alpha$ and $\mathbf{x}(t+1)$ is chosen to minimize it. Thus, the next lemma follows.

\begin{Lem}\label{lm:decision-update-inequality}
Let $\mathbf{z}\in \mathcal{X}_{0}$ be arbitrary. For all $t\geq 1$, Algorithm \ref{alg:new-alg} guarantees
\begin{align*}
&V [\nabla f^{t}(\mathbf{x}(t))]\tran [\mathbf{x}(t+1) - \mathbf{x}(t)]+  \sum_{k=1}^{m} Q_{k}(t)[\nabla g^{t}_{k}(\mathbf{x}(t))]\tran [\mathbf{x}(t+1) - \mathbf{x}(t)] + \alpha \Vert \mathbf{x}(t+1) - \mathbf{x}(t)\Vert^{2} \\
\leq& V [\nabla f^{t}(\mathbf{x}(t))]\tran [\mathbf{z} - \mathbf{x}(t)]+  \sum_{k=1}^{m} Q_{k}(t)[\nabla g^{t}_{k}(\mathbf{x}(t))]\tran [\mathbf{z} - \mathbf{x}(t)] + \alpha \Vert \mathbf{z} - \mathbf{x}(t)\Vert^{2} - \alpha\Vert \mathbf{z} - \mathbf{x}(t+1)\Vert^{2}. 
\end{align*}
\end{Lem}
The next corollary follows by taking $\mathbf{z}= \mathbf{x}(t)$ in Lemma \ref{lm:decision-update-inequality}. 
\begin{Cor} \label{cor:x-difference-bound}
For all $t\geq 1$, Algorithm \ref{alg:new-alg} guarantees $\Vert \mathbf{x}(t+1) - \mathbf{x}(t)\Vert \leq \frac{VD_1}{2\alpha} + \frac{\sqrt{m}D_2}{2\alpha} \Vert \mathbf{Q}(t)\Vert$.
\end{Cor}
\begin{proof}
%See Appendix \ref{app:pf-cor-x-difference-bound}.
Fix $t\geq 1$. Note that $\mathbf{x}(t)\in \mathcal{X}_0$. Taking $\mathbf{z} = \mathbf{x}(t)$ in Lemma \ref{lm:decision-update-inequality} yields
\begin{align*}
&V [\nabla f^{t}(\mathbf{x}(t))]\tran [\mathbf{x}(t+1) - \mathbf{x}(t)]+  \sum_{k=1}^{m} Q_{k}(t) [\nabla g^{t}_{k}(\mathbf{x}(t))]\tran [\mathbf{x}(t+1) - \mathbf{x}(t)] + \alpha \Vert \mathbf{x}(t+1) - \mathbf{x}(t)\Vert^{2} \\
\leq& - \alpha \Vert \mathbf{x}(t) - \mathbf{x}(t+1)\Vert^{2}. 
\end{align*}
Rearranging terms and cancelling common terms yields
\begin{align*}
2\alpha \Vert \mathbf{x}(t+1) - \mathbf{x}(t)\Vert^2 \leq& -V [\nabla f^{t}(\mathbf{x}(t))]\tran [\mathbf{x}(t+1) - \mathbf{x}(t)] - \sum_{k=1}^{m} Q_{k}(t)[\nabla g^{t}_{k}(\mathbf{x}(t))]\tran [\mathbf{x}(t+1) - \mathbf{x}(t)]  \\
\overset{(a)}{\leq}& V \Vert \nabla f^{t}(\mathbf{x}(t))\Vert \Vert \mathbf{x}(t+1) - \mathbf{x}(t)\Vert + \Vert \mathbf{Q}(t)\Vert \sqrt{\sum_{k=1}^m \Vert \nabla g^{t}_{k}(\mathbf{x}(t))\Vert^2 \Vert \mathbf{x}(t+1) - \mathbf{x}(t)\Vert^2}\\
\overset{(b)}{\leq} & V D_1 \Vert \mathbf{x}(t+1) - \mathbf{x}(t)\Vert + \sqrt{m}D_2 \Vert \mathbf{Q}(t)\Vert \Vert \mathbf{x}(t+1) - \mathbf{x}(t)\Vert
\end{align*}
where (a) follows by the Cauchy-Schwarz inequality (note that the second term on the right side applies the Cauchy-Schwarz inequality twice); and (b) follows from Assumption \ref{as:basic}. 

Thus, we have 
\begin{align*}
\Vert \mathbf{x}(t+1) - \mathbf{x}(t)\Vert \leq \frac{VD_1}{2\alpha} + \frac{\sqrt{m}D_2}{2\alpha} \Vert \mathbf{Q}(t)\Vert.
\end{align*}
\end{proof}

The next corollary follows directly from Lemma \ref{lm:queue-constraint-inequality} and Corollary \ref{cor:x-difference-bound} and shows that constraint violations are ultimately bounded by sequence $\Vert \mathbf{Q}(t)\Vert, t\in\{1,2,\ldots, T+1\}$.
\begin{Cor}\label{cor:queue-constraint-inequality}
For any $T\geq 1$, Algorithm \ref{alg:new-alg} guarantees 
$$\sum_{t=1}^{T} g^t_{k}(\mathbf{x}(t)) \leq   \Vert \mathbf{Q}(T+1)\Vert+  \frac{VTD_1D_2}{2\alpha} + \frac{\sqrt{m}D_2^2}{2\alpha}\sum_{t=1}^{T} \Vert \mathbf{Q}(t)\Vert \quad ,  \forall k\in\{1,2,\ldots,m\}$$
 where $D_1$ and $D_{2}$ are defined in Assumption \ref{as:basic}.
\end{Cor}
This corollary further justifies why Algorithm \ref{alg:new-alg} intends to minimize drift $\Delta(t)$. Recall that controlled drift can often lead to boundedness of a stochastic process as illustrated in the next section. Thus, the intuition of minimizing drift $\Delta(t)$ is to yield small $\Vert \mathbf{Q}(t)\Vert$ bounds. 

\section{ Expected Performance Analysis of Algorithm \ref{alg:new-alg}}

This section shows that if we choose $V = \sqrt{T}$ and $\alpha = T$ in Algorithm \ref{alg:new-alg}, then  both expected regret and expected constraint violations are $O(\sqrt{T})$.

\subsection{A Drift Lemma for Stochastic Processes}

Let $\{Z(t), t\geq 0\}$ be a discrete time stochastic process adapted\footnote{Random variable $Y$ is said to be adapted to $\sigma$-algebra $\mathcal{F}$ if $Y$ is $\mathcal{F}$-measurable. In this case, we often write $Y\in \mathcal{F}$. Similarly, random process $\{Z(t)\}$ is adapted to filtration $\{\mathcal{F}(t)\}$ if $Z(t)\in \mathcal{F}(t), \forall t$. See e.g. \cite{book_DurrettProbabilityTE}.}  to a filtration $\{\mathcal{F}(t), t\geq 0\}$. For example,  $Z(t)$ can be a random walk, a Markov chain or a martingale. Drift analysis is the method of deducing properties, e.g., recurrence, ergodicity, or boundedness, about $Z(t)$ from its drift $\mathbb{E}[Z(t+1) - Z(t) | \mathcal{F}(t)]$. See \cite{book_StochasticProcess_Doob,Hajek82} for more discussions or applications on drift analysis.  This paper proposes a new drift analysis lemma for stochastic processes as follows:

\begin{Lem} \label{lm:drift-random-process-bound}
Let $\{Z(t), t\geq 1\}$ be a discrete time stochastic process adapted to a filtration $\{\mathcal{F}(t), t\geq 1\}$.  Suppose there exists an integer $t_{0}>0$, real constants $\theta\in \mathbb{R}$, $\delta_{\max} > 0$, and $0< \zeta \leq \delta_{\max}$ such that 
\vskip -20pt
\begin{align} 
\vert Z(t+1) - Z(t) \vert \leq& \delta_{\max}, \\
\mathbb{E}[ Z(t+t_{0}) - Z(t) | \mathcal{F}(t)] \leq & \left\{ \begin{array}{cc} t_{0}\delta_{\max}, &\text{if}~ Z(t) < \theta \\  -t_{0}\zeta , &\text{if}~ Z(t) \geq \theta \end{array}\right.. \label{eq:stochastic-process-drift-condition}
\end{align}
\vskip -10pt
hold for all $t\in \{1,2,\ldots\}$. Then, the following holds

\begin{enumerate}[leftmargin=20pt]
\item $\mathbb{E}[Z(t)] \leq  \theta + t_{0}\frac{4\delta_{\max}^{2}}{\zeta} \log\big[1 + \frac{8\delta_{\max}^{2}}{\zeta^{2}}e^{\zeta/(4\delta_{\max})}\big], \forall t\in\{1,2,\ldots\}.$
\item For any constant $0<\mu<1$, we have $\text{Pr}(Z(t) \geq z) \leq \mu, \forall t\in\{1,2,\ldots\}$ where $z = \theta + t_{0}\frac{4\delta_{\max}^{2}}{\zeta} \log\big[1 + \frac{8\delta_{\max}^{2}}{\zeta^{2}}e^{\zeta/(4\delta_{\max})}\big] +  t_{0}\frac{4\delta_{\max}^{2}}{\zeta} \log(\frac{1}{\mu})$.
\end{enumerate}
\end{Lem}
\begin{proof}
See Appendix \ref{app:pf-lm-random-process-bound}.    
\end{proof}

The above lemma provides both expected and high probability bounds for stochastic processes based on a drift condition. It will be used to establish upper bounds of virtual queues $\Vert \mathbf{Q}(t) \Vert$, which further leads to expected and high probability constraint performance bounds of our algorithm. For a given stochastic process $Z(t)$, it is possible to show the drift condition \eqref{eq:stochastic-process-drift-condition} holds for multiple $t_0$ with different $\zeta$ and $\theta$. In fact, we will show in Lemma \ref{lm:queue-decrease-condition} that $\Vert \mathbf{Q}(t)\Vert$ yielded by Algorithm \ref{alg:new-alg} satisfies \eqref{eq:stochastic-process-drift-condition} for any integer $t_0 > 0$ by selecting $\zeta$ and $\theta$ according to $t_0$.  One-step drift conditions, corresponding to the special case $t_0=1$ of Lemma \ref{lm:drift-random-process-bound}, have been previously considered in \cite{Hajek82, Neely15INFOCOM}.  However, Lemma \ref{lm:drift-random-process-bound} (with general $t_0>0$) allows us to choose the best $t_0$ in performance analysis such that sublinear regret and constraint violation bounds are possible.

\subsection{Expected Constraint Violation Analysis}

Define filtration $\{\mathcal{W}(t), t\geq 0\}$ with $\mathcal{W}(0) = \{\emptyset, \Omega\}$ and $\mathcal{W}(t) = \sigma (\omega(1), \ldots, \omega(t))$ being the $\sigma$-algebra generated by random samples  $\{\omega(1), \ldots, \omega(t)\}$ up to round $t$. From the update rule in Algorithm \ref{alg:new-alg}, we observe that $\mathbf{x}(t+1)$ is a deterministic function of $f^t(\cdot), \mathbf{g}(\cdot; \omega(t))$ and $\mathbf{Q}(t)$ where $\mathbf{Q}(t)$ is further a deterministic function of $\mathbf{Q}(t-1), \mathbf{g}(\cdot; \omega(t-1))$, $\mathbf{x}(t)$ and $\mathbf{x}(t-1)$. By inductions, it is easy to show that $\sigma(\mathbf{x}(t)) \subseteq \mathcal{W}(t-1)$ and $\sigma(\mathbf{Q}(t))\subseteq \mathcal{W}(t-1)$ for all $t\geq 1$ where $\sigma(Y)$ denotes the $\sigma$-algebra generated by random variable $Y$. For fixed $t\geq 1$, since  $\mathbf{Q}(t)$ is fully determined by $\omega(\tau), \tau\in\{1, 2, \ldots, t-1\}$ and $\omega(t)$ are i.i.d., we know $\mathbf{g}^{t}(\mathbf{x})$ is independent of $\mathbf{Q}(t)$.  This is formally summarized in the next lemma.

\begin{Lem}\label{lm:queue-independence}
If $\mathbf{x}^{\ast}\in \mathcal{X}_{0}$  satisfies $\tilde{\mathbf{g}}(\mathbf{x}^{\ast})= \mathbb{E}_{\omega} [\mathbf{g}(\mathbf{x}^{\ast}; \omega)] \leq \mathbf{0}$, then Algorithm \ref{alg:new-alg} guarantees:

\begin{align}
\mathbb{E}[Q_{k}(t)g_{k}^{t}(\mathbf{x}^{\ast})] \leq 0, \forall k\in\{1,2,\ldots, m\}, \forall t\geq 1. \label{eq:expectation-is-0}
\end{align}
\end{Lem}

\begin{proof}~
Fix $k\in\{1,2,\ldots, m\}$ and $t\geq 1$. Since $g_{k}^{t}(\mathbf{x}^{\ast}) = g_k(\mathbf{x}^\ast; \omega(t))$ is independent of $Q_{k}(t)$, which is determined by $\{\omega(1), \ldots, \omega(t-1)\}$, it follows that  $\mathbb{E}[Q_{k}(t)g_{k}^{t}(\mathbf{x}^{\ast})] = \mathbb{E}[Q_{k}(t)] \mathbb{E}[g_{k}^{t}(\mathbf{x}^{\ast})] \overset{(a)}{\leq}0$, where (a) follows from the fact that $\mathbb{E}[g_{k}^{t}(\mathbf{x}^{\ast})] \leq 0$ and $Q_{k}(t)\geq 0$.
\end{proof}

To establish a bound on constraint violations, by Corollary \ref{cor:queue-constraint-inequality}, it suffices to derive upper bounds for $\Vert \mathbf{Q}(t)\Vert$. In this subsection, we derive upper bounds for $\Vert \mathbf{Q}(t)\Vert$ by applying the  drift lemma (Lemma \ref{lm:drift-random-process-bound}) developed at the beginning of this section.  The next lemma shows that random process  $ Z(t) = \Vert \mathbf{Q}(t)\Vert$ satisfies the conditions in Lemma \ref{lm:drift-random-process-bound}.

\begin{Lem}\label{lm:queue-decrease-condition}
Let $t_0>0$ be an arbitrary integer. At each round $t\in\{1,2,\ldots,\}$ in Algorithm \ref{alg:new-alg},  the following holds

\begin{align*}
\big \vert \Vert \mathbf{Q}(t+1)\Vert - \Vert \mathbf{Q}(t)\Vert \big\vert \leq& G+\sqrt{m} D_2 R, \quad \text{and}
\end{align*}

\begin{align*}
\mathbb{E}[\Vert \mathbf{Q}(t+t_0)\Vert - \Vert \mathbf{Q}(t)\Vert \big|  \mathcal{W}(t-1)] \leq & \left\{ \begin{array}{cc} t_0(G+\sqrt{m} D_2 R), &\text{if}~ \Vert \mathbf{Q}(t)\Vert <  \theta \\  - t_0 \frac{\epsilon}{2}, &\text{if}~ \Vert \mathbf{Q}(t)\Vert \geq  \theta \end{array}\right.,
\end{align*}
where $\theta = \frac{\epsilon}{2}t_0+ (G+\sqrt{m}D_2R)t_0 + \frac{2\alpha R^2}{t_0 \epsilon}+ \frac{2VD_{1} R +  (G+\sqrt{m}D_{2}R)^{2}}{\epsilon}$, $m$ is the number of constraint functions; $D_{1}, D_{2}, G$ and $R$ are defined in Assumption \ref{as:basic}; and $\epsilon$ is defined in Assumption \ref{as:interior-point}. (Note that $\epsilon < G$ by the definition of $G$.)
\end{Lem}

\begin{proof}
See Appendix \ref{app:pf-queue-decrease-condition}. 
\end{proof}
Lemma \ref{lm:queue-decrease-condition} allows us to apply Lemma \ref{lm:drift-random-process-bound}   to random process $Z(t) = \Vert\mathbf{Q}(t)\Vert$ and obtain $\mathbb{E}[\Vert \mathbf{Q}(t)\Vert] = O(\sqrt{T}), \forall t$ by taking $t_0 = \lceil \sqrt{T}\rceil$, $V=\sqrt{T}$ and $\alpha = T$, where $\lceil \sqrt{T}\rceil$ represents the smallest integer no less than $\sqrt{T}$. By Corollary \ref{cor:queue-constraint-inequality}, this further implies the expected constraint violation bound $\mathbb{E}[ \sum_{t=1}^{T}g_{k}(\mathbf{x}(t))] \leq O(\sqrt{T})$ as  summarized in the next theorem.

\begin{Thm}[Expected Constraint Violation Bound]\label{thm:constraint-bound}
If $V = \sqrt{T}$ and $\alpha = T$ in Algorithm \ref{alg:new-alg}, then for all $T\geq 1$, we have
\begin{align}
\mathbb{E}[ \sum_{t=1}^{T}g^t_{k}(\mathbf{x}(t))] \leq O(\sqrt{T}), \forall k\in\{1,2,\ldots,m\}.
\end{align}
where the expectation is taken with respect to all $\omega(t)$.
\end{Thm}
\begin{proof}
Define random process $Z(t) = \Vert \mathbf{Q}(t)\Vert$ and filtration $\mathcal{F}(t) = \mathcal{W}(t-1)$. Note that $Z(t)$ is adapted to $\mathcal{F}(t)$. By Lemma \ref{lm:queue-decrease-condition}, $Z(t)$ satisfies the conditions in Lemma \ref{lm:drift-random-process-bound} with $\delta_{\max} = G+\sqrt{m}D_2R$, $\zeta = \frac{\epsilon}{2}$ and $\theta = \frac{\epsilon}{2}t_0+ (G+\sqrt{m}D_2R)t_0 + \frac{2\alpha R^2}{t_0 \epsilon}+ \frac{2VD_{1} R +  (G+\sqrt{m}D_{2}R)^{2}}{\epsilon}$. Thus, by part (1) of Lemma \ref{lm:drift-random-process-bound}, for all $t\in\{1,2,\ldots\}$, we have 
\begin{align*}
\mathbb{E}[\Vert \mathbf{Q}(t)\Vert] \leq \frac{\epsilon}{2}t_0+ (G+\sqrt{m}D_2R)t_0 + \frac{2\alpha R^2}{t_0 \epsilon}+ \frac{2VD_{1} R +  (G+\sqrt{m}D_{2}R)^{2}}{\epsilon}+ t_0B
\end{align*}
where $B = \frac{8(G+\sqrt{m}D_{2}R)^{2}}{\epsilon}\log[1+ \frac{32[G+\sqrt{m}D_{2}R]^{2}}{\epsilon^{2}} e^{\epsilon/[8(G+\sqrt{m}D_{2}R)]}]$ is an absolute constant irrelevant to algorithm parameters. Taking $t_0 = \lceil \sqrt{T}\rceil $, $V =\sqrt{T}$ and $\alpha = T$, for all $t\in\{1,2,\ldots\}$, we have
\begin{align*}
\mathbb{E}[\Vert \mathbf{Q}(t)\Vert] \leq &\frac{\epsilon}{2} \lceil \sqrt{T} \rceil + (G+\sqrt{m}D_2R)  \lceil \sqrt{T} \rceil + \frac{2 T R^2}{\lceil \sqrt{T} \rceil \epsilon}+ \frac{2\sqrt{T}D_{1} R +  (G+\sqrt{m}D_{2}R)^{2}}{\epsilon}+ \lceil \sqrt{T} \rceil B \\
=&  O(\sqrt{T}).
\end{align*}
Fix $T\geq 1$. By Corollary \ref{cor:queue-constraint-inequality} (with $V =\sqrt{T}$ and $\alpha = T$) , we have 
\begin{align*}
\sum_{t=1}^{T} g^t_{k}(\mathbf{x}(t)) \leq  \Vert \mathbf{Q}(T+1)\Vert +  \frac{\sqrt{T}D_1D_2}{2} + \frac{\sqrt{m}D_2^2}{2 T}\sum_{t=1}^{T} \Vert \mathbf{Q}(t)\Vert, \forall k\in\{1,2,\ldots,m\}.
\end{align*}
Taking expectations on both sides and substituting $\mathbb{E}[\Vert \mathbf{Q}(t)\Vert] =  O(\sqrt{T}), \forall t$ into it yields 
\begin{align*}
\mathbb{E}[\sum_{t=1}^{T} g^t_{k}(\mathbf{x}(t))] \leq  O(\sqrt{T}).
\end{align*}
\end{proof}

\subsection{Expected Regret Analysis}

The next lemma refines Lemma \ref{lm:decision-update-inequality}  and is useful to analyze the regret.
\begin{Lem}\label{lm:deterministic-regret-bound}
Let $\mathbf{z}\in \mathcal{X}_{0}$ be arbitrary. For all $T\geq 1$, Algorithm \ref{alg:new-alg} guarantees
\begin{align}
\sum_{t=1}^{T}f^{t}(\mathbf{x}(t))\leq \sum_{t=1}^{T}f^{t}(\mathbf{z}) + \underbrace{\frac{\alpha}{V} R^{2} + \frac{VD_{1}^{2}}{4\alpha} T + \frac{1}{2}[G+\sqrt{m}D_{2}R]^{2} \frac{T}{V}}_{\text{(I)}} + \underbrace{\frac{1}{V}\sum_{t=1}^{T} \big[\sum_{k=1}^{m} Q_{k}(t) g_{k}^{t}(\mathbf{z})\big]}_{\text{(II)}} \label{eq:deterministic-regret-bound}
\end{align}
where $m$ is the number of constraint functions; and $D_{1}, D_{2}, G$ and $R$ are defined in Assumption \ref{as:basic}.
\end{Lem}
\begin{proof}
%See Appendix \ref{app:pf-lm-deterministic-regret-bound}.
Fix $t\geq 1$.  By Lemma \ref{lm:decision-update-inequality}, we have
\begin{align*}
&V [\nabla f^{t}(\mathbf{x}(t))]\tran [\mathbf{x}(t+1) - \mathbf{x}(t)]+  \sum_{k=1}^{m} Q_{k}(t) [\nabla g^{t}_{k}(\mathbf{x}(t))]\tran [\mathbf{x}(t+1) - \mathbf{x}(t)] + \alpha \Vert \mathbf{x}(t+1) - \mathbf{x}(t)\Vert^{2} \\
\leq& V [\nabla f^{t}(\mathbf{x}(t))]\tran [\mathbf{z} - \mathbf{x}(t)]+  \sum_{k=1}^{m} Q_{k}(t) [\nabla g^{t}_{k}(\mathbf{x}(t))]\tran [\mathbf{z} - \mathbf{x}(t)]  + \alpha [\Vert \mathbf{z} - \mathbf{x}(t)\Vert^{2} - \Vert \mathbf{z} - \mathbf{x}(t+1)\Vert^{2}]. 
\end{align*}
Adding constant $Vf^{t}(\mathbf{x}(t)) + \sum_{k=1}^{m} Q_k(t) g_k^{t}(\mathbf{x}(t))$ on both sides; and noting that 
$f^{t}(\mathbf{x}(t)) + [\nabla f^{t}(\mathbf{x}(t))]\tran [\mathbf{z} - \mathbf{x}(t)] \leq f^{t}(\mathbf{z})$ and $g_k^{t}(\mathbf{x}(t))+[\nabla g^{t}_{k}(\mathbf{x}(t))]\tran [\mathbf{z} - \mathbf{x}(t)] \leq g_k^{t}(\mathbf{z})$ by convexity yields
\begin{align}
&Vf^{t}(\mathbf{x}(t)) +V [\nabla f^{t}(\mathbf{x}(t))]\tran [\mathbf{x}(t+1) - \mathbf{x}(t)]+  \sum_{k=1}^{m} Q_{k}(t)\big[ g^{t}_{k}(\mathbf{x}(t)) + [\nabla g^{t}_{k}(\mathbf{x}(t))]\tran [\mathbf{x}(t+1) - \mathbf{x}(t)] \big] \nonumber \\ & \qquad+ \alpha \Vert \mathbf{x}(t+1) - \mathbf{x}(t)\Vert^{2} \nonumber \\
\leq& V [\nabla f^{t}(\mathbf{z})]+  \sum_{k=1}^{m} Q_{k}(t)g_{k}^{t}(\mathbf{z}) + \alpha [\Vert \mathbf{z} - \mathbf{x}(t)\Vert^{2} - \Vert \mathbf{z} - \mathbf{x}(t+1)\Vert^{2}].  \label{eq:pf-dpp-bound-eq1}
\end{align}
By Lemma \ref{lm:drift}, we have 
\begin{align}
\Delta(t) \leq  \sum_{k=1}^{m} Q_{k}(t)\big[ g_{k}^{t}(\mathbf{x}(t)) + [\nabla g^{t}_{k}(\mathbf{x}(t))]\tran [\mathbf{x}(t+1) - \mathbf{x}(t)]\big] +  \frac{1}{2}[G+\sqrt{m}D_{2}R]^{2}. \label{eq:pf-dpp-bound-eq2}
\end{align}
Summing \eqref{eq:pf-dpp-bound-eq1} and \eqref{eq:pf-dpp-bound-eq2}, cancelling common terms  and rearranging terms yields

\begin{align}
Vf^{t}(\mathbf{x}(t)) \leq&Vf^{t}(\mathbf{z}) -\Delta(t) +  \sum_{k=1}^{m} Q_{k}(t) g_{k}^{t}(\mathbf{z})+ \alpha [\Vert \mathbf{z} - \mathbf{x}(t)\Vert^{2}- \Vert \mathbf{z} - \mathbf{x}(t+1)\Vert^{2}] \nonumber\\ & -  V [\nabla f^{t}(\mathbf{x}(t))]\tran [\mathbf{x}(t+1) - \mathbf{x}(t)] - \alpha \Vert \mathbf{x}(t+1) - \mathbf{x}(t)\Vert^{2} +   \frac{1}{2}[G+\sqrt{m}D_{2}R]^{2} \label{eq:pf-dpp-bound-eq3}
\end{align}
Note that 
\begin{align}
&-V [\nabla f^{t}(\mathbf{x}(t))]\tran [\mathbf{x}(t+1) - \mathbf{x}(t)] - \alpha \Vert \mathbf{x}(t+1) - \mathbf{x}(t)\Vert^{2} \nonumber \\
\overset{(a)}{\leq}& V \Vert \nabla f^{t}(\mathbf{x}(t))\Vert \Vert\mathbf{x}(t+1) - \mathbf{x}(t)\Vert - \alpha \Vert \mathbf{x}(t+1) - \mathbf{x}(t)\Vert^{2} \nonumber\\
\overset{(b)}{\leq}& V D_1 \Vert\mathbf{x}(t+1) - \mathbf{x}(t)\Vert - \alpha \Vert \mathbf{x}(t+1) - \mathbf{x}(t)\Vert^{2} \nonumber\\
= &  -\alpha \big[ \Vert \mathbf{x}(t+1) - \mathbf{x}(t)\Vert - \frac{V D_{1}}{2\alpha}\big]^{2} + \frac{V^{2}D_{1}^{2}}{4\alpha} \nonumber\\
\leq&  \frac{V^{2}D_{1}^{2}}{4\alpha} \label{eq:pf-dpp-bound-eq4}
\end{align}
where (a) follows from the Cauchy-Schwartz inequality; and (b) follows from Assumption \ref{as:basic}.
Substituting \eqref{eq:pf-dpp-bound-eq4} into \eqref{eq:pf-dpp-bound-eq3} yields
 \begin{align*}
 Vf^{t}(\mathbf{x}(t)) \leq Vf^{t}(\mathbf{z}) -\Delta(t) +  \sum_{k=1}^{m} Q_{k}(t) g_{k}^{t}(\mathbf{z})+ \alpha [\Vert \mathbf{z} - \mathbf{x}(t)\Vert^{2}- \Vert \mathbf{z} - \mathbf{x}(t+1)\Vert^{2}] + \frac{V^{2}D_{1}^{2}}{4\alpha}  + \frac{1}{2}[G+\sqrt{m}D_{2}R]^{2}.
\end{align*}
Summing over $t\in\{1,2,\ldots,T\}$ yields
\begin{align*}
V\sum_{t=1}^{T}f^{t}(\mathbf{x}(t)) \leq & V\sum_{t=1}^{T}f^{t}(\mathbf{z}) - \sum_{t=1}^{T}\Delta(t) +  \alpha \sum_{t=1}^{T}[\Vert \mathbf{z}- \mathbf{x}(t)\Vert^{2} - \Vert \mathbf{z} - \mathbf{x}(t+1)\Vert^{2}]  + \frac{V^{2}D_{1}^{2}}{4\alpha} T \\ &+ \frac{1}{2}[G+\sqrt{m}D_{2}R]^{2} T + \sum_{t=1}^{T} \big[\sum_{k=1}^{m} Q_{k}(t) g_{k}^{t}(\mathbf{z})\big]\\
\overset{(a)}{=}& V\sum_{t=1}^{T}f^{t}(\mathbf{z}) + L(1) -L(T+1)+ \alpha \Vert \mathbf{z} - \mathbf{x}(1)\Vert^{2} -  \alpha \Vert \mathbf{z} - \mathbf{x}(T+1)\Vert^{2} + \frac{V^{2}D_{1}^{2}}{4\alpha} T \\&+ \frac{1}{2}[G+\sqrt{m}D_{2}R]^{2} T + \sum_{t=1}^{T} \big[\sum_{k=1}^{m} Q_{k}(t) g_{k}^{t}(\mathbf{z})\big]\\
\overset{(b)}{\leq} & V\sum_{t=1}^{T}f^{t}(\mathbf{z}) + \alpha R^{2} + \frac{V^{2}D_{1}^{2}}{4\alpha} T + \frac{1}{2}[G+\sqrt{m}D_{2}R]^{2} T + \sum_{t=1}^{T} \big[\sum_{k=1}^{m} Q_{k}(t) g_{k}^{t}(\mathbf{z})\big]. 
\end{align*}
where (a) follows by recalling that $\Delta(t) = L(t+1) -L(t)$; and (b) follows because $\Vert \mathbf{z} - \mathbf{x}(1)\Vert\leq R$ by Assumption \ref{as:basic},  $L(1)=\frac{1}{2} \Vert \mathbf{Q}(1)\Vert^2 = 0$ and $L(T+1) = \frac{1}{2}\Vert \mathbf{Q}(T+1)\Vert^{2} \geq 0$.

Dividing both sides by $V$ yields the desired result.

\end{proof}
Note that if we take $V=\sqrt{T}$ and $\alpha = T$, then term (I) in \eqref{eq:deterministic-regret-bound} is $O(\sqrt{T})$. Recall that the expectation of term (II) in \eqref{eq:deterministic-regret-bound} with $\mathbf{z} = \mathbf{x}^\ast$ is non-positive by Lemma \ref{lm:queue-independence}. The expected regret bound of Algorithm \ref{alg:new-alg} follows by taking expectations on both sides of \eqref{eq:deterministic-regret-bound} and is summarized in the next theorem.

\begin{Thm}[Expected Regret Bound] \label{thm:regret-bound}
Let $\mathbf{x}^{\ast}\in \mathcal{X}_{0}$ be any fixed solution that satisfies $\tilde{\mathbf{g}}(\mathbf{x}^{\ast}) \leq \mathbf{0}$, e.g., $\mathbf{x}^{\ast} = \argmin_{\mathbf{x}\in \mathcal{X}} \sum_{t=1}^{T} f^{t}(\mathbf{x})$.   If $V = \sqrt{T}$ and $\alpha = T$ in Algorithm \ref{alg:new-alg}, then for all $T\geq 1$, 

\begin{align*}
\mathbb{E}[\sum_{t=1}^{T}f^{t}(\mathbf{x}(t))] \leq \mathbb{E}[\sum_{t=1}^{T}f^{t}(\mathbf{x}^{\ast})] +  O(\sqrt{T}).
\end{align*}
where the expectation is taken with respect to all $\omega(t)$.
\end{Thm}
\begin{proof}
Fix $T\geq 1$. Taking $\mathbf{z} = \mathbf{x}^\ast$ in Lemma \ref{lm:deterministic-regret-bound} yields
\begin{align*}
\sum_{t=1}^{T}f^{t}(\mathbf{x}(t))\leq \sum_{t=1}^{T}f^{t}(\mathbf{x}^{\ast}) + \frac{\alpha}{V} R^{2} + \frac{VD_{1}^{2}}{4\alpha} T + \frac{1}{2}[G+\sqrt{m}D_{2}R]^{2} \frac{T}{V} + \frac{1}{V}\sum_{t=1}^{T} \big[\sum_{k=1}^{m} Q_{k}(t) g_{k}^{t}(\mathbf{x}^{\ast})\big].
\end{align*}
Taking expectations on both sides and using \eqref{eq:expectation-is-0} yields 
\begin{align*}
\sum_{t=1}^{T}\mathbb{E}[f^{t}(\mathbf{x}(t))] \leq \sum_{t=1}^{T} \mathbb{E}[f^{t}(\mathbf{x}^{\ast})] +   \frac{\alpha}{V}R^{2}+ \frac{D_{1}^{2}}{4} \frac{V}{\alpha}T + \frac{1}{2}[G+\sqrt{m}D_{2}R]^{2} \frac{T}{V}.
\end{align*}
Taking $V = \sqrt{T}$ and $\alpha = T$ yields
\begin{align*}
\sum_{t=1}^{T}\mathbb{E}[f^{t}(\mathbf{x}(t))] \leq \sum_{t=1}^{T}\mathbb{E}[f^{t}(\mathbf{x}^{\ast})] +  O(\sqrt{T}).
\end{align*}
\end{proof}

\subsection{Special Case Performance Guarantees}
Theorems \ref{thm:constraint-bound} and \ref{thm:regret-bound} provide expected performance guarantees of Algorithm \ref{alg:new-alg} for OCO with stochastic constraints.  The results further imply the performance guarantees in the following special cases:

\begin{itemize}[leftmargin=20pt]
\item  {\bf OCO with long term constraints}: In this case, $g_k(\mathbf{x}; \omega(t))\equiv g_k(\mathbf{x})$ and there is no randomness. Thus, the expectations in Theorems \ref{thm:constraint-bound} and \ref{thm:regret-bound} disappear.  For this problem, Algorithm \ref{alg:new-alg} can achieve $O(\sqrt{T})$ (deterministic) regret and $O(\sqrt{T})$ (deterministic) constraint violations.
\item  {\bf Stochastic constrained convex optimization}: Note that i.i.d. time-varying $f(\mathbf{x}; \omega(t))$ is a special case of arbitrarily-varying $f^{t}(\mathbf{x})$ as considered in our OCO setting. Thus, Theorems \ref{thm:constraint-bound} and \ref{thm:regret-bound} still hold when Algorithm \ref{alg:new-alg} is applied to stochastic constrained convex optimization.  That is, $\sum_{t=1}^T \mathbb{E}[f^t(\mathbf{x}(t))] \leq \sum_{t=1}^T \mathbb{E}[f^{t}(\mathbf{x}^\ast)] + O(\sqrt{T})$ and $\sum_{t=1}^{T} \mathbb{E}[g_k^t (\mathbf{x}(t))] \leq O(\sqrt{T}),\forall k\in\{1,2,\ldots,n\}$. This online performance guarantee also implies Algorithm \ref{alg:new-alg} can be used a (batch) offline algorithm with $O(1/\sqrt{T})$ convergence for stochastic constrained convex optimization. That is, after running Algorithm \ref{alg:new-alg} for $T$ slots, if we use  $\overline{\mathbf{x}}(T) =\frac{1}{T}\sum_{t=1}^{T}\mathbf{x}(t)$ as a fixed solution, then $\mathbb{E}[f(\overline{\mathbf{x}}(T); \omega)] = \mathbb{E}[f^t(\overline{\mathbf{x}}(T))] \leq \mathbb{E}[f^{t}(\mathbf{x}^\ast)] + O(\frac{1}{\sqrt{T}})$ and $\mathbb{E}[g_k(\overline{\mathbf{x}}(T);\omega)] = \mathbb{E}[g_k^t (\overline{\mathbf{x}}(T))] \leq O(\frac{1}{\sqrt{T}}),\forall k\in\{1,2,\ldots,n\}$ with $t\geq T+1$ by the i.i.d. property of each $f^t$ and $g^t$ and Jensen's inequality.  If we use Algorithm \ref{alg:new-alg} as a (batch) offline algorithm, its performance ties with the algorithm developed in \cite{Lan16arXiv}, which is by design a (batch) offline algorithm and can only solve stochastic optimization with a single constraint function.

\item  {\bf Deterministic constrained convex optimization}:  Similarly to OCO with long term constraints,  the expectations in Theorems \ref{thm:constraint-bound} and \ref{thm:regret-bound} disappear in this case since $f^t(\mathbf{x})\equiv f(\mathbf{x})$ and $g_k(\mathbf{x}; \omega(t))\equiv g_k(\mathbf{x})$.  If we use $\overline{\mathbf{x}}(T) =\frac{1}{T}\sum_{t=1}^{T}\mathbf{x}(t)$ as the solution, then $f(\overline{\mathbf{x}}(T)) \leq f(\mathbf{x}^\ast) +O(\frac{1}{\sqrt{T}})$ and $g_k(\overline{\mathbf{x}}(T)) \leq O(\frac{1}{\sqrt{T}})$, which follows by dividing inequalities in Theorems \ref{thm:constraint-bound} and \ref{thm:regret-bound} by $T$ on both sides and applying Jensen's inequality.  Thus, Algorithm \ref{alg:new-alg} solves deterministic constrained convex optimization with $O(\frac{1}{\sqrt{T}})$ convergence.
\end{itemize}

\section{High Probability Performance Analysis}

This section  shows that if we choose $V = \sqrt{T}$ and $\alpha = T$ in Algorithm \ref{alg:new-alg}, then for any $0<\lambda<1$, with probability at least $1-\lambda$, regret is $O(\sqrt{T}\log(T)\log^{1.5}(\frac{1}{\lambda}))$ and constraint violations are $O\big(\sqrt{T}\log(T)\log(\frac{1}{\lambda})\big)$.
\vspace{-0.5em}
\subsection{High Probability Constraint Violation Analysis}
Similarly to the expected constraint violation analysis, we can use part (2) of the new drift lemma (Lemma \ref{lm:drift-random-process-bound}) to obtain a high probability bound of $\Vert \mathbf{Q}(t)\Vert$, which together with Corollary \ref{cor:queue-constraint-inequality} leads to a high probability constraint violation bound summarized in Theorem \ref{thm:high-prob-constraint-bound}.

\begin{Thm}[High Probability Constraint Violation Bound] \label{thm:high-prob-constraint-bound}
Let $0<\lambda<1$ be arbitrary. If $V = \sqrt{T}$ and $\alpha = T$ in Algorithm \ref{alg:new-alg}, then for all $T\geq 1$ and all $k\in\{1,2,\ldots,m\}$, we have
\begin{align*}
\text{Pr}\Big( \sum_{t=1}^{T}g_{k}(\mathbf{x}(t)) \leq O\big(\sqrt{T}\log(T)\log(\frac{1}{\lambda})\big) \Big) \geq 1-\lambda.
\end{align*}
\end{Thm}
\begin{proof}
%See Appendix \ref{app:pf-thm-high-prob-constraint-bound}.
Define random process $Z(t) = \Vert \mathbf{Q}(t)\Vert, \forall t\in\{1,2,\ldots\}$. By Lemma \ref{lm:queue-decrease-condition}, $Z(t)$ satisfies the conditions in Lemma \ref{lm:drift-random-process-bound} with $\delta_{\max} = G+\sqrt{m}D_2R$, $\zeta = \frac{\epsilon}{2}$ and $$\theta = \frac{\epsilon}{2}t_0+ (G+\sqrt{m}D_2R)t_0 + \frac{2\alpha R^2}{t_0 \epsilon}+ \frac{4VD_{1} R +  [G+\sqrt{m}D_{2}R]^{2}}{\epsilon}.$$  Fix $T\geq 1$ and $0<\lambda <1$. Taking $\mu = \lambda/(T+1)$ in part (2) of Lemma \ref{lm:drift-random-process-bound} yields $$\text{Pr}(\Vert \mathbf{Q}(t)\Vert \geq \gamma) \leq \frac{\lambda}{T+1}, \forall t\in\{1,2,\ldots, T+1\},$$ where $\gamma = \frac{\epsilon}{2}t_0+ (G+\sqrt{m}D_2R)t_0 + \frac{2\alpha R^2}{t_0 \epsilon}+ \frac{2VD_{1} R +  [G+\sqrt{m}D_{2}R]^{2}}{\epsilon}+ t_0B + t_0 \frac{8[G+\sqrt{m}D_{2}R]^{2}}{\epsilon} \log(\frac{T+1}{\lambda})$ with $B = \frac{8[G+\sqrt{m}D_{2}R]^{2}}{\epsilon}\log[1+ \frac{32[G+\sqrt{m}D_{2}R]^{2}}{\epsilon^{2}} e^{\epsilon/[8(G+\sqrt{m}D_{2}R)]}]$ being an absolute constant irrelevant to algorithm parameters. 

By union bounds, we have
\begin{align*}
\text{Pr}(\Vert \mathbf{Q}(t)\Vert \geq \gamma \text{~for some~} t\in\{ 1,2,\ldots, T+1\}) \leq \lambda.
\end{align*}
This implies 
\begin{align}
\text{Pr}(\Vert \mathbf{Q}(t)\Vert \leq \gamma \text{~for all~} t\in\{ 1,2,\ldots, T+1\}) \geq 1-\lambda. \label{eq:pf-thm-high-prob-constraint-bound-eq1}
\end{align}
Taking $t_0 = \lceil \sqrt{T}\rceil$, $V = \sqrt{T}$ and $\alpha = T$ yields 
\begin{align}
\gamma = O(\sqrt{T}\log(T)) + O(\sqrt{T}\log(\frac{1}{\lambda})) = O(\sqrt{T}\log(T)\log(\frac{1}{\lambda})) \label{eq:pf-thm-high-prob-constraint-bound-eq2}
\end{align} 
Recall that by Corollary \ref{cor:queue-constraint-inequality} (with $V=\sqrt{T}$ and $\alpha = T$), for all $k\in\{1,2,\ldots, m\}$, we have 
\begin{align}
\sum_{t=1}^{T} g_{k}(\mathbf{x}(t)) \leq  \Vert \mathbf{Q}(T+1)\Vert +  \frac{\sqrt{T}D_1D_2}{2} + \frac{\sqrt{m}D_2^2}{2T}\sum_{t=1}^{T} \Vert \mathbf{Q}(t)\Vert. \label{eq:pf-thm-high-prob-constraint-bound-eq3}
\end{align}
It follows from \eqref{eq:pf-thm-high-prob-constraint-bound-eq1}-\eqref{eq:pf-thm-high-prob-constraint-bound-eq3} that $$\text{Pr}\big( \sum_{t=1}^{T} g_{k}(\mathbf{x}(t)) \leq O(\sqrt{T}\log(T)\log(\frac{1}{\lambda})) \geq 1-\lambda.$$

\end{proof}

%\subsection{Generalized Hoeffding-Azuma Inequality}
\subsection{High Probability Regret Analysis}

To obtain a high probability regret bound from Lemma \ref{lm:deterministic-regret-bound}, it remains to derive a high probability bound of term (II) in \eqref{eq:deterministic-regret-bound} with $\mathbf{z} = \mathbf{x}^\ast$.  The main challenge is that term (II) is a supermartingale with unbounded differences (due to the possibly unbounded virtual queues $Q_k(t)$).  Most  concentration inequalities, e.g., the Hoeffding-Azuma inequality, used in high probability performance analysis of online algorithms are restricted to martingales/supermartingales with bounded differences. See for example \cite{book_PredictionLearningGames,Bartlett08COLT,Mahdavi12JMLR}.  The following lemma considers supermartingales with unbounded differences. Its proof uses the truncation method to construct an auxiliary well-behaved supermargingale. Similar proof  techniques are previously used in \cite{Vu02Concentration,Tao15AnnalsProbability} to prove different concentration inequalities for supermartingales/martingales with unbounded differences.

\begin{Lem}\label{lm:extended-Azuma-inequality}
Let $\{Z(t), t\geq 0\}$ be a supermartingale adapted to a filtration $\{\mathcal{F}(t), t\geq 0\}$ with $Z(0) = 0$ and $\mathcal{F}(0) =\{\emptyset, \Omega\}$, i.e., $\mathbb{E}[Z(t+1)|\mathcal{F}(t)] \leq Z(t), \forall t\geq 0$. If there exits constant $c>0$ such that $ \{ \vert Z(t+1) - Z(t) \vert > c\}\subseteq \{Y(t) > 0\}, \forall t\geq 0$, where each $Y(t)$ is adapted to $\mathcal{F}(t)$ and $\text{Pr}(Y(t) > 0)\leq p(t), \forall t\geq 0$. Then, for all $z>0$, we have
\begin{align*}
\text{Pr}(Z(t) \geq z) \leq e^{-z^2/(2tc^2)} + \sum_{\tau=0}^{t-1} p(\tau), \forall t\geq 1.
\end{align*}
\end{Lem}
\begin{proof}
See Appendix \ref{app:pf-lm-extended-Azuma-inequality}.
\end{proof}

Note that if $p(t)=0, \forall t\geq 0$, then $Z(t)$ is a supermartingale with differences bounded by $c$ and $\text{Pr}(\{ \vert Z(t+1) - Z(t) \vert> c\}) = 0, \forall t\geq 0$. In this case, Lemma \ref{lm:extended-Azuma-inequality} reduces to the conventional Hoeffding-Azuma inequality.

The next theorem summarizes the high probability regret performance of Algorithm \ref{alg:new-alg} and follows from Lemmas \ref{lm:drift-random-process-bound}-\ref{lm:extended-Azuma-inequality}  . 

\begin{Thm}[High Probability Regret Bound] \label{thm:high-prob-regret-bound}
Let $\mathbf{x}^{\ast}\in \mathcal{X}_{0}$ satisfy $\tilde{\mathbf{g}}(\mathbf{x}^{\ast}) \leq \mathbf{0}$, e.g., $\mathbf{x}^{\ast} = \argmin_{\mathbf{x}\in \mathcal{X}} \sum_{t=1}^{T} f^{t}(\mathbf{x})$.  Let $0<\lambda<1$ be arbitrary. If $V = \sqrt{T}$ and $\alpha = T$ in Algorithm \ref{alg:new-alg}, then for all $T\geq 1$, we have
\begin{align*}
\text{Pr}\Big(\sum_{t=1}^{T}f^{t}(\mathbf{x}(t))) \leq \sum_{t=1}^{T}f^{t}(\mathbf{x}^{\ast}) +  O(\sqrt{T}\log(T)\log^{1.5}(\frac{1}{\lambda}))\Big) \geq 1-\lambda.
\end{align*}
\end{Thm}
\vspace{-1.4em}
\begin{proof}
See Appendix \ref{app:thm-high-prob-regret-bound}.
\end{proof}
\section{Experiment: Online Job Scheduling in Distributed Data Centers}
Consider a geo-distributed data center infrastructure consisting of one front-end job router and $100$ geographically distributed servers, which are located at $10$ different zones to form $10$ clusters ($10$ servers in each cluster). See Fig. \ref{fig}(a) for an illustration. The front-end job router receives job tasks and schedules them to different servers to fulfill the service. To serve the assigned jobs, each server purchases power (within its capacity) from its zone market.  Electricity market prices can vary significantly across time and zones. For example, see Fig. \ref{fig}(b) for a $5$-minute average electricity price trace (between $05/01/2017$ and $05/10/2017$) at New York zone CENTRL \cite{NYISO}. This problem is to schedule jobs and control power levels at each server in real time such that all incoming jobs are served and electricity cost is minimized. In our experiment, each server power is adjusted every $5$ minutes, which is called a slot. (In practice, server power can not be adjusted too frequently due to hardware restrictions and configuration delay.)  Let $\mathbf{x}(t) = [x_1(t), \ldots, x_{100}(t)]$ be the power vector at slot $t$, where each $x_i (t)$ must be chosen from an interval $[x_i^{\min}, x_i^{\max}]$  restricted by the hardware, and service rate at each server $i$ satisfy $\mu_i (t) = h_i (x_i(t))$, where $h_i(\cdot)$ is an increasing concave function. At each slot $t$, the job router schedules $\mu_i(t)$ amount of jobs to server $i$.  The electricity cost at slot $t$ is $f^t(\mathbf{x}(t)) = \sum_{i=1}^{100} c_i(t) x_i(t)$ where $c_i(t)$ is the electricity price at server $i$'s zone.  Use $c_i(t)$ from real-world $5$-minute average electricity price data at $10$ different zones in New York city between $05/01/2017$ and $05/10/2017$ obtained from NYISO \cite{NYISO}.  At each slot $t$, the incoming job is given by $\omega(t)$ and satisfies a Poisson distribution. Note that the amount of incoming jobs and electricity price $c_i(t)$ are unknown to us at the beginning of each slot $t$ but can be observed at the end of each slot.  This is an example of OCO with stochastic constraints, where we aim to minimize the electricity cost subject to the constraint that incoming jobs must be served in time. In particular, at each round $t$, we receive loss function $f^t(\mathbf{x}(t))$ and constraint function $g^t(\mathbf{x}(t)) = \omega(t) - \sum_{i=1}^{100} h_i(x_i(t))$. 

We compare our proposed algorithm with $3$ baselines:  (1) best fixed decision in hindsight; (2) react \cite{gandhi2012sleep} and (3) low-power \cite{qureshi2009cutting}.  Both ``react" and ``low-power" are popular power control strategies used in distributed data centers.  See Supplement \ref{app:experiment} for more details of  these $2$ baselines and our experiment. Fig. \ref{fig}(c)(d) plot the performance of $4$ algorithms, where the running average is the time average up to the current slot. Fig. \ref{fig}(c) compares electricity cost while Fig. \ref{fig}(d) compares unserved jobs. (Unserved jobs accumulated if the service rate provided by an algorithm is less than the job arrival rate, i.e., the stochastic constraint is violated.) Fig. \ref{fig}(c)(d)  show that our proposed algorithm performs closely to the best fixed decision in hindsight over time, both in electricity cost and constraint violations. `React" performs well in serving job arrivals but yields larger electricity cost, while ``low-power" has low electricity cost but fails to serve job arrivals. 
\vspace{-0.5em}

\begin{figure*}[ht!] 
    \centering
    \begin{subfigure}[t]{0.5\textwidth}
        \centering
        \includegraphics[height=7cm] {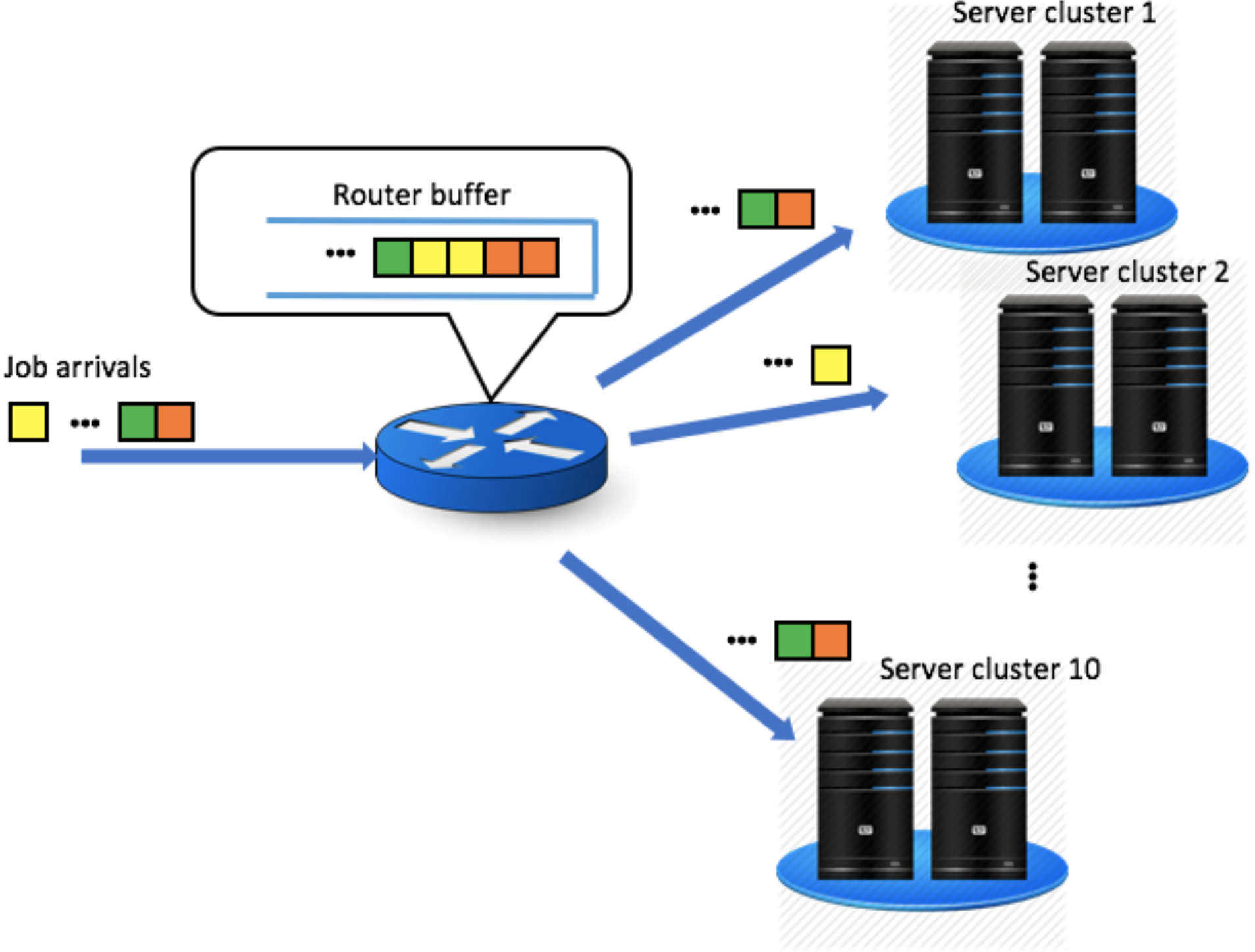}
        \caption{}
    \end{subfigure}%
    ~ 
    \begin{subfigure}[t]{0.5\textwidth}
        \centering
        \includegraphics[height=7cm] {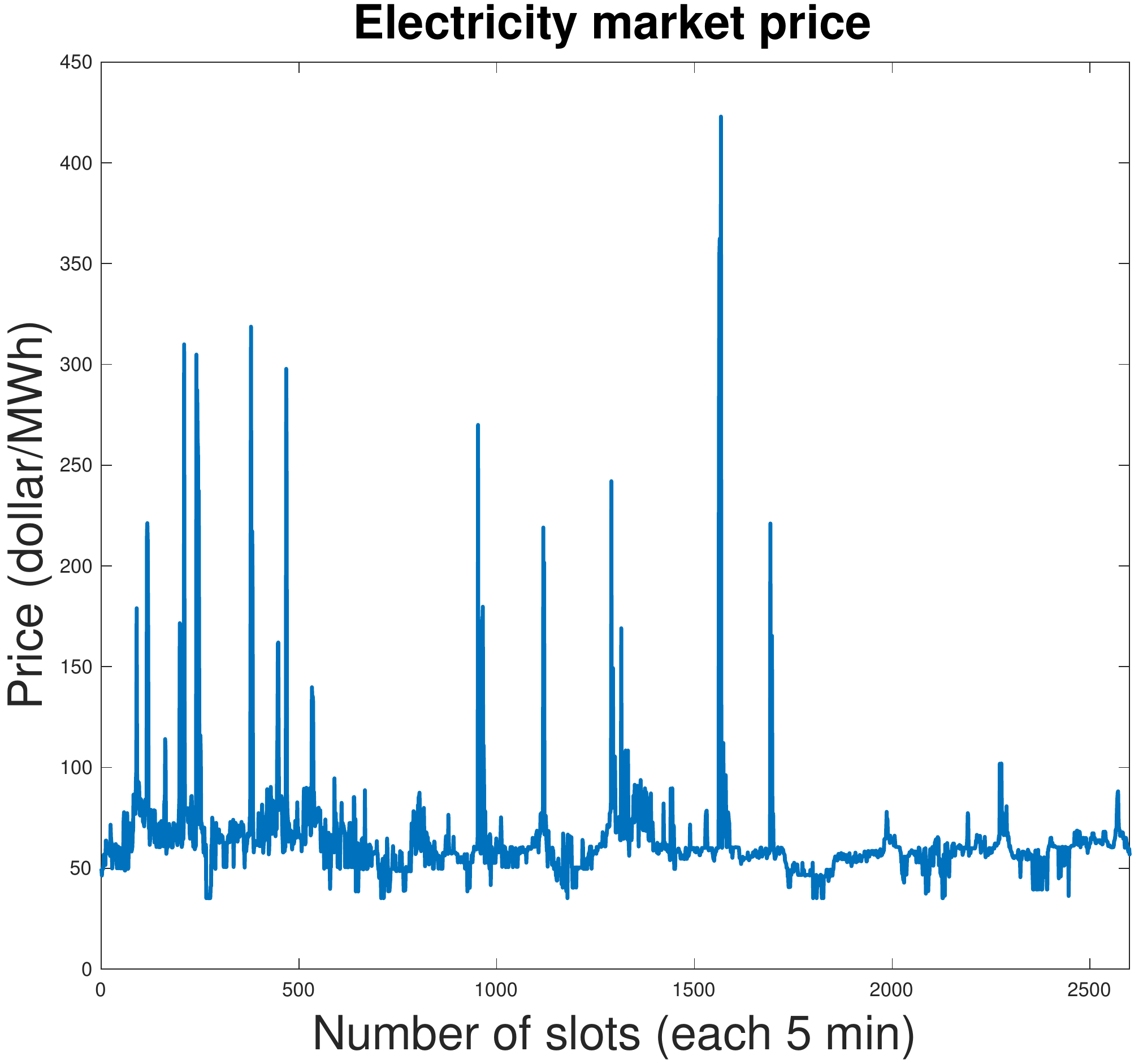}
        \caption{}
    \end{subfigure}
    ~
    \begin{subfigure}[t]{0.5\textwidth}
        \centering
        \includegraphics[height=7cm] {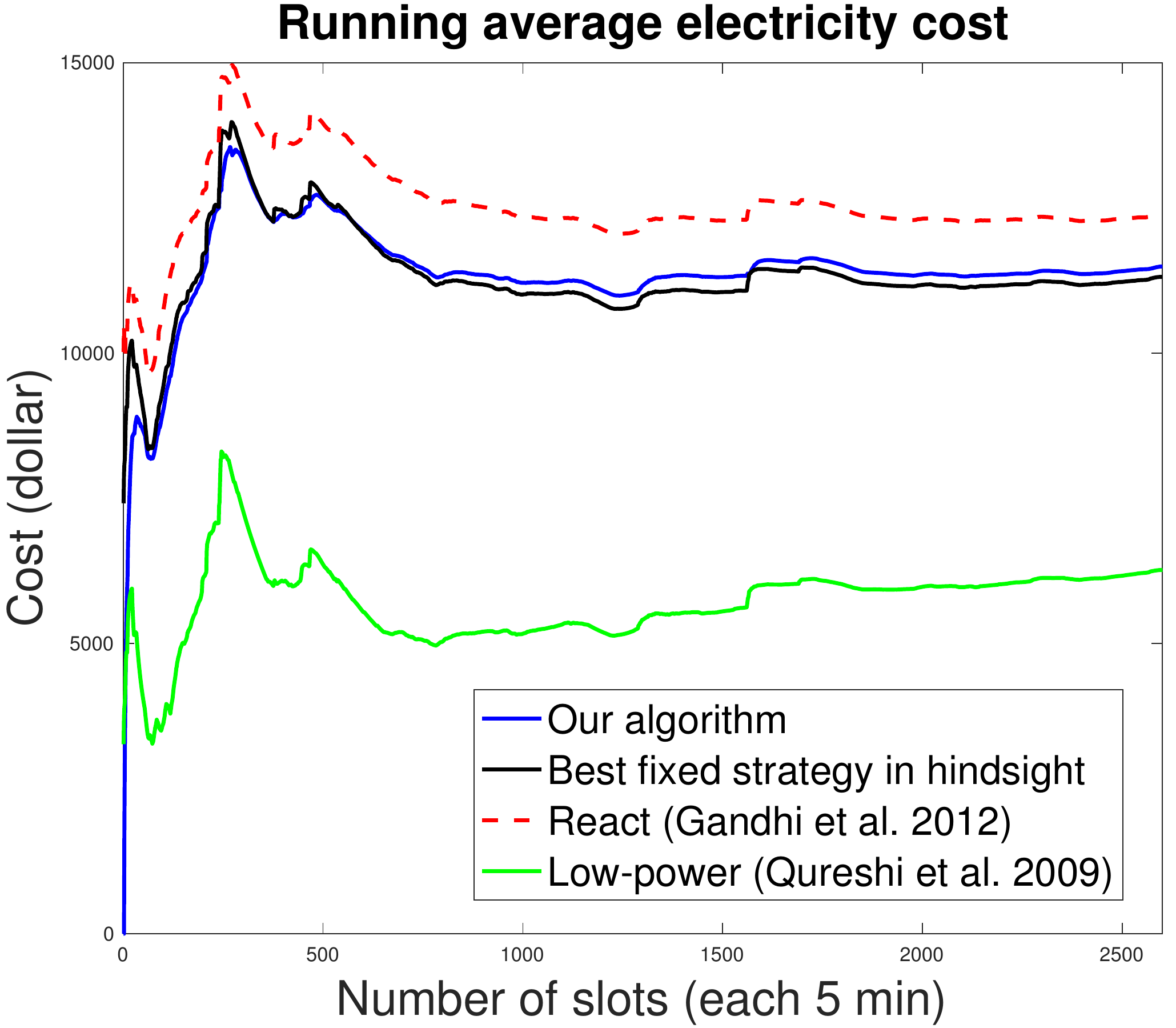}
        \caption{}
    \end{subfigure}%
    ~      
    \begin{subfigure}[t]{0.5\textwidth}
        \centering
        \includegraphics[height=7cm] {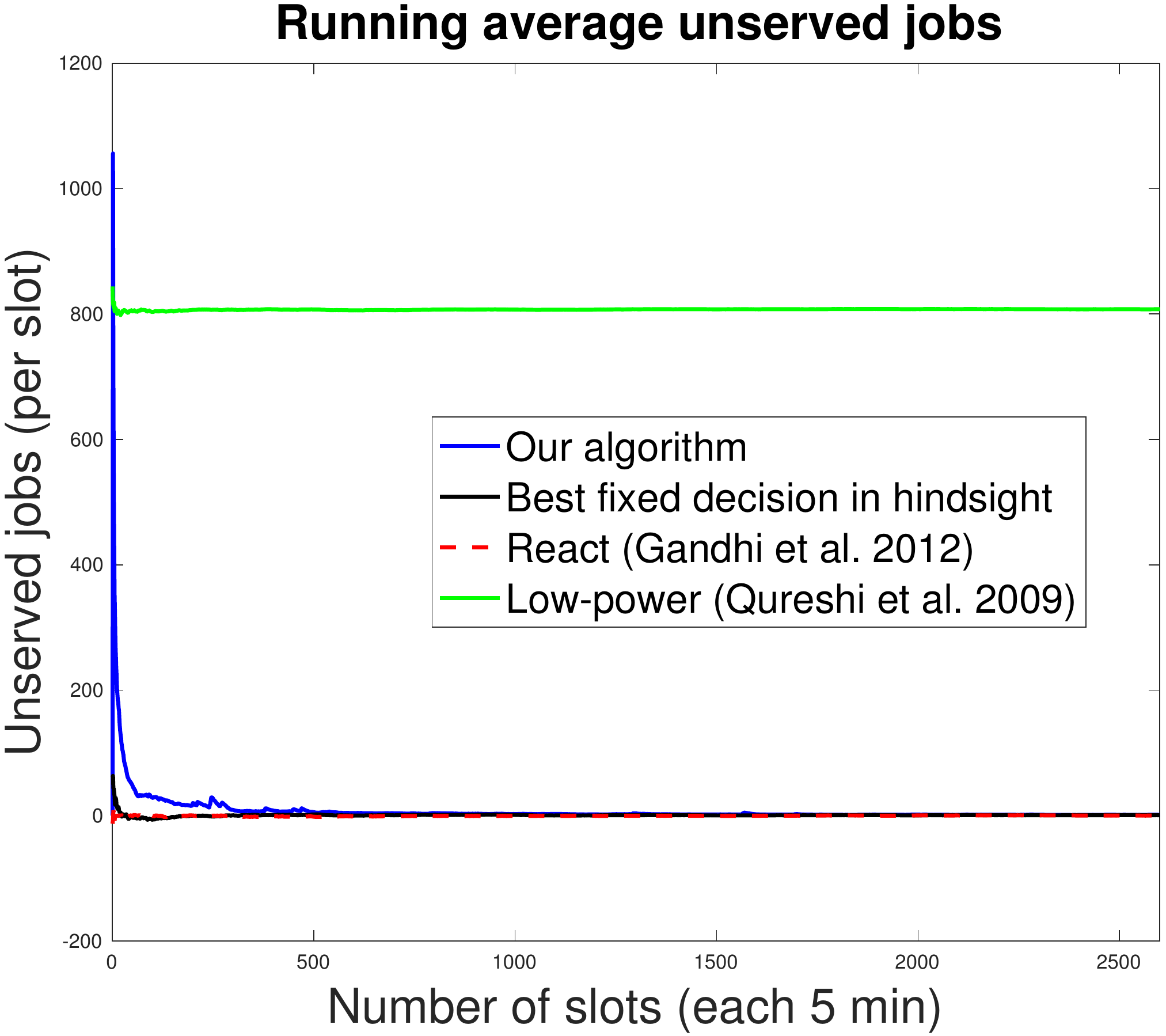}
        \caption{}
    \end{subfigure}   
    \caption{(a) Geo-distributed data  center infrastructure;  (b) Electricity market prices at zone CENTRAL New York; (c) Running average electricity cost; (d) Running average unserved jobs.}\label{fig}
\end{figure*}

\section{Conclusion}

This paper studies OCO with stochastic constraints and proposes a novel learning algorithm that guarantees $O(\sqrt{T})$ expected regret and constraint violations and $O(\sqrt{T}\log(T))$ high probability regret and constraint violations.

\newpage

\appendices

\section{Proof of Lemma \ref{lm:drift-random-process-bound}} \label{app:pf-lm-random-process-bound}

In this proof, we first  establish an upper bound of $\mathbb{E}[e^{r Z(t)}]$ for some constant $r>0$. Part (1) of this lemma follows  by applying Jensen's inequality since $e^{rx}$ is convex with respect to $x$ when $r>0$.  Part (2) of this lemma follows directly from Markov's inequality.

The following fact is useful in the proof. 
\begin{Fact}\label{fact:exp-inequality}
$e^{x} \leq 1+ x + 2x^{2}$ for any $|x|\leq 1$.
\end{Fact}

\begin{proof}
By Taylor's expansion, we known for any $x\in \mathbb{R}$,  there exists a point $\hat{x}$ in between $0$ and $x$ such that $e^{x} = 1+ x + e^{\hat{x}}\frac{x^2}{2}$. (Note that the value of $\hat{x}$ depends on $x$ and if $x>0$, then $\hat{x}\in (0, x)$; if $x<0$, then $\hat{x}\in (x,0)$; and if $x=0$, then $\hat{x}=x$. )  Since $|x|\leq 1$, we have $e^{\hat{x}} \leq e \leq 4$. Thus,  $e^{x} \leq 1+ x+ 2x^2$ for any $|x|\leq 1$. 
\end{proof}

The next lemma provides an upper bound of $\mathbb{E}[e^{r Z(t)}]$ with constant $r = \frac{\zeta}{4\delta_{\max}^{2}}<1$.
\begin{Lem}
\begin{align*}
\mathbb{E}[e^{rZ(t)}] \leq \frac{e^{rt_{0}\delta_{\max}}}{1-\rho} e^{r\theta} + \rho^{\lceil \frac{t}{t_{0}}\rceil}, \forall t\in\{0,1,\ldots\},
\end{align*}
where $r = \frac{\zeta}{4t_{0}\delta_{\max}^{2}}$, $\rho = 1- \frac{\zeta^{2}}{8\delta_{\max}^{2}} = 1 - \frac{rt_{0}\zeta}{2}$ and where $\lceil x \rceil$ denotes the smallest integer that is no less than $x$. 
\end{Lem}
\begin{proof}
Since $0<\zeta <\delta_{\max}$, we have $0<\rho<1 < e^{r\delta_{\max}}$. Define $\eta(t) = Z(t+t_{0}) - Z(t)$. Note that $|\eta(t)| \leq t_{0}\delta_{\max}, \forall t\geq 0$ and  $|r\eta(t) | \leq \frac{\zeta}{4 t_{0}\delta_{\max}^2} t_{0}\delta_{\max} = \frac{\zeta}{4\delta_{\max}} \leq 1$. Then, 
\begin{align}
e^{rZ(t+t_{0})} =& e^{rZ(t)} e^{r\eta(t)}  \label{eq:pf-lm-random-process-bound-eq0}\\
\overset{(a)}{\leq}& e^{rZ(t)} [1 +r\eta(t) + 2r^2 t_{0}^{2}\delta_{\max}^2] \nonumber \\
\overset{(b)}{=}& e^{rZ(t)} [1 +r\eta(t) + \frac{1}{2}r t_{0}\zeta],  \label{eq:pf-lm-random-process-bound-eq1}
\end{align}
where (a) follows from Fact \ref{fact:exp-inequality} by noting that $|r\eta(t)|\leq 1$ and $|\eta(t)|\leq t_{0}\delta_{\max}$; and (b) follows by substituting $r = \frac{\zeta}{4t_{0}\delta_{\max}^{2}}$ into a single $r$ of the term $2r^{2}t_{0}^{2}\delta_{\max}^{2}$.

Next, consider the cases $Z(t)\geq \theta $ and $Z(t)<\theta$, separately.
\begin{itemize}[leftmargin=20pt]
\item Case $Z(t)\geq \theta$:  Taking conditional expectations on both sides  of \eqref{eq:pf-lm-random-process-bound-eq1} yields:
\begin{align*}
\mathbb{E}[e^{rZ(t+t_{0})} | Z(t)] \leq& \mathbb{E} [e^{rZ(t)} (1 +r\eta(t) + \frac{1}{2}rt_{0}\zeta) | Z(t)] \\
\overset{(a)}{\leq}& e^{rZ(t)} \big[ 1 -rt_{0}\zeta + \frac{1}{2}rt_{0}\zeta\big]\\
=& e^{rZ(t)} \big[ 1- \frac{rt_{0}\zeta}{2}\big ]\\
\overset{(b)}{=}& \rho e^{rZ(t)}.
\end{align*}
where (a) follows from the fact that $\mathbb{E}[Z(t+t_{0}) - Z(t)|\mathcal{F}(t)] \leq -t_{0}\zeta$ when $Z(t) \geq \theta$; and (b) follows from the fact that $\rho = 1 - \frac{rt_{0}\zeta}{2}$. 
\item Case $Z(t) < \theta$: Taking conditional expectations on both sides  of \eqref{eq:pf-lm-random-process-bound-eq0} yields:
\begin{align*}
\mathbb{E}[e^{rZ(t+t_{0})} | Z(t) ] = &\mathbb{E}[e^{rZ(t)}e^{r\eta(t)} | Z(t)] \\
=& e^{rZ(t)} \mathbb{E}[e^{r\eta(t)}| Z(t)]\\
\overset{(a)}{\leq} & e^{rt_{0}\delta_{\max}}e^{rZ(t)}, 
\end{align*}
where (a) follows from the fact that $\eta(t) \leq t_{0}\delta_{\max}$.
\end{itemize}
Putting two cases together yields:
\begin{align}
\mathbb{E}[e^{rZ(t+t_{0})}] \overset{(a)}{=} & \text{Pr}(Z(t) \geq \theta) \mathbb{E}[e^{rZ(t+t_{0})} | Z(t) \geq \theta] + \text{Pr}(Z(t) < \theta) \mathbb{E}[e^{rZ(t+t_{0})} | Z(t) < \theta] \nonumber\\
\overset{(b)}{\leq}& \rho \mathbb{E}[e^{rZ(t)}| Z(t) \geq \theta] \text{Pr}(Z(t) \geq \theta) + e^{rt_{0}\delta_{\max}} \mathbb{E}[e^{rZ(t)}| Z(t) < \theta] \text{Pr}(Z(t) < \theta) \nonumber\\\
\overset{(c)}{=}& \rho \mathbb{E}[e^{rZ(t)}] + [e^{r t_{0}\delta_{\max}} - \rho] \mathbb{E}[e^{rZ(t)}| Z(t) < \theta] \text{Pr}(Z(t) < \theta) \nonumber\\\
\overset{(d)}{\leq} & \rho \mathbb{E}[e^{rZ(t)}] + [e^{rt_{0}\delta_{\max}} - \rho] e^{r\theta}\nonumber\\
\leq & \rho \mathbb{E}[e^{rZ(t)}] + e^{rt_{0}\delta_{\max}} e^{r\theta}, \label{eq:pf-lm-random-process-bound-eq2}
\end{align}
where (a) follows by the definition of expectations; (b) follows from the results in the above two cases; (c) follows from the fact that $\mathbb{E}[e^{rZ(t)}] = \text{Pr}(Z(t)\geq \theta) \mathbb{E}[e^{rZ(t)}| Z(t)\geq \theta] +  \text{Pr}(Z(t)< \theta) \mathbb{E}[e^{rZ(t)}| Z(t)< \theta]$; and (d) follow from the fact that $e^{rt_{0}\delta_{\max}}> \rho$.  

Now, we prove $\mathbb{E}[e^{rZ(t)}] \leq \frac{e^{rt_{0}\delta_{\max}}}{1-\rho} e^{r\theta} + \rho^{\lceil \frac{t}{t_{0}}\rceil}, \forall t\geq 0$ , where $\lceil x \rceil$ denotes the smallest integer that is no less than $x$, by inductions. Since $Z(\tau) \leq \tau \delta_{\max}, \forall \tau \geq  0$, it follows that $\mathbb{E}[e^{rZ(\tau)}] \leq e^{r \tau \delta_{\max}} \leq e^{r t_{0} \delta_{\max}} \leq \frac{e^{r t_{0}\delta_{\max}}}{1-\rho} e^{r\theta} + \rho^{\lceil \frac{\tau}{t_{0}}\rceil}, \forall \tau\in\{1,\ldots, t_{0}\}$, where the last inequality follows because $\frac{e^{r\theta}}{1-\rho} \geq 1$ and $ \rho^{\lceil \frac{\tau}{t_{0}}\rceil} \geq 1$. 

Fix $\tau \in\{1,\ldots, t_{0}\}$. Assume $\mathbb{E}[e^{rZ(nt_{0} + \tau)}] \leq \frac{e^{r t_{0}\delta_{\max}}}{1-\rho} e^{r\theta} + \rho^{n+1}$ holds. Note that the base case $n=0$ is just proven above. Consider $\mathbb{E}[e^{rZ((n+1)t_{0}+\tau)}]$.  By \eqref{eq:pf-lm-random-process-bound-eq2}, we have
\begin{align*}
\mathbb{E}[e^{rZ((n+1)t_{0} + \tau)}]  \leq& \rho \mathbb{E}[e^{rZ(nt_{0}+\tau)}] + e^{rt_{0}\delta_{\max}} e^{r\theta} \\
\leq & \rho [\frac{e^{rt_{0}\delta_{\max}}}{1-\rho} e^{r\theta} + \rho^{n+1}] +  e^{rt_{0}\delta_{\max}} e^{r\theta}\\
\leq &  e^{rt_{0}\delta_{\max}} e^{r\theta} [\frac{\rho}{1-\rho} + 1] + \rho^{n+2}\\
= &\frac{e^{rt_{0}\delta_{\max}} }{1-\rho} e^{r\theta}+ \rho^{n+2}.
\end{align*}
Thus, this lemma follows by inductions.
\end{proof}

By this  lemma, for all $t\in\{0,1,\ldots\}$, we have 
\begin{align}
\mathbb{E}[e^{rZ(t)}] \leq& \frac{e^{rt_{0}\delta_{\max}}}{1-\rho} e^{r\theta} + \rho^{\lceil \frac{t}{t_{0}}\rceil} \nonumber\\
\overset{(a)}{\leq} & \frac{e^{rt_{0}\delta_{\max}} }{1-\rho} e^{r\theta} + 1 \nonumber\\
\overset{(b)}{\leq} &  \frac{e^{rt_{0}\delta_{\max}}}{1-\rho} e^{r\theta} + e^{r\theta} \nonumber\\
= & e^{r\theta} \big[\frac{e^{rt_{0}\delta_{\max}}}{1-\rho} +1\big], \label{eq:pf-lm-random-process-bound-eq3}
\end{align}
where (a) follows from the fact that $0<\rho<1$; and (b) follows from the facts that $r>0$ and $\theta>0$.

{\bf Proof of Part (1):} Note that $e^{rx}$ is convex with respect to $x$ when $r>0$. By Jensen's inequality,
\begin{align}
e^{r\mathbb{E}[Z(t)]} \leq &\mathbb{E}[e^{rZ(t)}] \nonumber \\
\overset{(a)}{\leq} & e^{r\theta} \big[\frac{e^{rt_{0}\delta_{\max}}}{1-\rho} +1\big], \label{eq:pf-lm-random-process-bound-eq4}
\end{align}
where (a) follows from \eqref{eq:pf-lm-random-process-bound-eq3}. 

Taking logarithm on both sides and dividing by $r$ yields:
\begin{align*}
\mathbb{E}[Z(t)] \leq& \theta + \frac{1}{r} \log\big[1 + \frac{e^{rt_{0}\delta_{\max}}}{1-\rho}\big] \\
\overset{(a)}{=}& \theta + t_{0}\frac{4\delta_{\max}^{2}}{\zeta} \log\big[1 + \frac{8\delta_{\max}^{2}e^{\zeta/(4\delta_{\max})}}{\zeta^{2}}\big],
\end{align*}
where (a) follows by recalling that $r = \frac{\zeta}{4t_{0}\delta_{\max}^{2}}$ and $\rho = 1- \frac{\zeta^{2}}{8\delta_{\max}^{2}}$.

{\bf Proof of Part (2):} 
Fix $z$. Note that 
\begin{align}
\text{Pr} (Z(t) > z) =& \text{Pr}(e^{rZ(t)} > e^{rz}) \nonumber \\
\overset{(a)}{\leq}& \frac{\mathbb{E}[e^{rZ(t)}]}{e^{rz}} \nonumber \\
\overset{(b)}{\leq}& e^{r\theta} e^{-rz} \big[\frac{e^{rt_{0}\delta_{\max}}}{1-\rho} +1\big] \nonumber \\
\overset{(c)}{=}& e^{\frac{\zeta}{4t_0 \delta_{\max}^2}(\theta - z)} \big[ 1 + \frac{8\delta_{\max}^{2}e^{\zeta/(4\delta_{\max})}}{\zeta^{2}}\big] \label{eq:pf-lm-random-process-bound-eq5}
\end{align}
where (a) follows from Markov's inequality; (b) follows from \eqref{eq:pf-lm-random-process-bound-eq3}; and (c) follows by recalling that $r = \frac{\zeta}{4t_{0}\delta_{\max}^{2}}$ and $\rho = 1- \frac{\zeta^{2}}{8\delta_{\max}^{2}}$.

Define $\mu = e^{\frac{\zeta}{4t_0 \delta_{\max}^2}(\theta - z)} \big[ 1 + \frac{8\delta_{\max}^{2}e^{\zeta/(4\delta_{\max})}}{\zeta^{2}}\big]$. It follows that if $$z =  \theta + t_{0}\frac{4\delta_{\max}^{2}}{\zeta} \log\big[1 + \frac{8\delta_{\max}^{2}e^{\zeta/(4\delta_{\max})}}{\zeta^{2}}\big] +  t_{0}\frac{4\delta_{\max}^{2}}{\zeta} \log(\frac{1}{\mu}),$$ then we have $\text{Pr}(Z(t) \geq z) \leq \mu$ by \eqref{eq:pf-lm-random-process-bound-eq5}.

\section{Proof of Lemma \ref{lm:queue-decrease-condition}} \label{app:pf-queue-decrease-condition}

The next lemma will be useful in our proof.
\begin{Lem} \label{lm:Slater-negative}
Let $\hat{\mathbf{x}}\in \mathcal{X}_{0}$ be a Slater point defined in Assumption \ref{as:interior-point}, i.e, $\tilde{g}_{k}(\hat{\mathbf{x}}) = \mathbb{E}_{\omega} [g_{k}(\hat{\mathbf{x}}; \omega)]\leq -\epsilon, \forall k\in\{1,2,\ldots, m\}$. Then
\begin{align*}
\mathbb{E}[\sum_{k=1}^m Q_k (t_{1}) g^{t_{1}}_{k}(\hat{\mathbf{x}}) | \mathcal{W}(t_{2})]  \leq  -\epsilon  \mathbb{E}[\Vert\mathbf{Q}(t_{1})\Vert | \mathcal{W}(t_{2})], \quad  \forall t_{2} \leq t_{1}-1
\end{align*}
where $\epsilon>0$ is defined in Assumption \ref{as:interior-point}.
\end{Lem}
\begin{proof} To prove this lemma, we first show that $\mathbb{E}[Q_{k}(t_{1}) g_{k}^{t_{1}}(\hat{\mathbf{x}}) | \mathcal{W}(t_{2})] \leq -\epsilon \mathbb{E}[Q_{k}(t_{1}) | \mathcal{W}(t_{2})], \forall k\in\{1,2,\ldots, m\}, \forall t_{2} \leq t_{1} -1$. Fix $k\in\{1,2,\ldots,m\}$. Note that $\mathbf{Q}(t_1) \in \mathcal{W}(t_1-1)$ and $g_{k}^{t_{1}} (\hat{\mathbf{x}})$ is independent of $\mathcal{W}(t_1-1) $. Further,  if $t_{2} \leq t_{1}-1$, then $ \mathcal{W}(t_2) \subseteq \mathcal{W}(t_1-1)$.  Thus, we have 
\begin{align*}
\mathbb{E}[Q_{k}(t_{1}) g_{k}^{t_{1}}(\hat{\mathbf{x}}) | \mathcal{W}(t_{2})]
\overset{(a)}{=}&\mathbb{E}\big[ \mathbb{E}[Q_{k}(t_{1}) g_{k}^{t_{1}}(\hat{\mathbf{x}}) | \mathcal{W}(t_1 -1)] | \mathcal{W}(t_{2})\big]\\
\overset{(b)}{=}&\mathbb{E}\big[ Q_{k}(t_{1}) \mathbb{E}[ g_{k}^{t_{1}}(\hat{\mathbf{x}})] | \mathcal{W}(t_{2})\big]\\
\overset{(c)}{=} &\mathbb{E}[g_{k}^{t_{1}}(\hat{\mathbf{x}})] \mathbb{E}[Q_{k}(t_{1})| \mathcal{W}(t_{2})]\\
\overset{(d)}{\leq}& -\epsilon \mathbb{E}[Q_{k}(t_{1})| \mathcal{W}(t_{2})]
\end{align*}
where (a) follows from iterated expectations; (b) follows because $g_{k}^{t_{1}} (\hat{\mathbf{x}})$ is independent of $\mathcal{W}(t_1-1) $ and $Q_k(t_1)\in \mathcal{W}(t_1 -1)$; (c) follows by extracting the constant $\mathbb{E}[g_{k}^{t_{1}}(\hat{\mathbf{x}})]$  and (d) follows from the assumption that $\hat{\mathbf{x}}$ is a Slater point,  $g^{t}(\cdot)$ are i.i.d. across $t$ and the fact that $Q_{k}(t)\geq 0$.

Now, summing over $m\in\{1,2,\ldots, m\}$ yields
\begin{align*}
\mathbb{E}[\sum_{k=1}^m Q_k (t_{1}) g^{t_{1}}_{k}(\hat{\mathbf{x}}) | \mathcal{W}(t_{2})]  \leq&  -\epsilon  \mathbb{E}[\sum_{k=1}^{m}Q_{k}(t_{1})| \mathcal{W}(t_{2})]\\
\overset{(a)}{\leq}& -\epsilon \mathbb{E}[ \Vert \mathbf{Q}(t_{1}) \Vert | \mathcal{W}(t_{2})]
\end{align*}
where (a) follows from the basic fact that $\sum_{k=1}^{m} a_{k} \geq \sqrt{\sum_{k=1}^{m} a_{k}^{2}}$ when $a_{k}
\geq 0, \forall k\in\{1,2,\ldots,m\}$.
\end{proof}

The bounded difference of $\vert \mathbf{Q}(t+1) - \mathbf{Q}(t)\vert$ follows directly from the virtual queue update equation \eqref{eq:queue-update} and is summarized in the next Lemma.

\begin{Lem}\label{lm:simple-queue-diff}
Let $\mathbf{Q}(t), t\in\{0,1,\ldots\}$ be the sequence generated by Algorithm \ref{alg:new-alg}. Then,
$$\Vert \mathbf{Q}(t)\Vert - G - \sqrt{m}D_{2} R\leq \Vert \mathbf{Q}(t+1)\Vert \leq \Vert \mathbf{Q}(t)\Vert + G, \forall t\geq 0.$$

\end{Lem}
\begin{proof}~

\begin{itemize}[leftmargin=10pt]
\item {\bf Proof of $\Vert \mathbf{Q}(t+1)\Vert \leq \Vert \mathbf{Q}(t)\Vert + G$}:

Fix $t\geq 0$ and $k\in\{1,2,\ldots,m\}$. The virtual queue update equation implies that 
\begin{align*}
Q_{k}(t+1) =& \max\{ Q_{k}(t) + g_{k}^{t}(\mathbf{x}(t)) + [\nabla g_{k}^{t}(\mathbf{x}(t))]\tran [\mathbf{x}(t+1) - \mathbf{x}(t)], 0\} \\
\overset{(a)}{\leq}&  \max\{ Q_{k}(t) + g_{k}^{t}(\mathbf{x}(t+1)), 0\},
\end{align*}
where (a) follows from the convexity of $g_{k}^{t}(\cdot)$. 

Note that $Q_{k}(t+1)\geq 0$ and recall the fact that if $0\leq a \leq \max\{b,0\}$, then $a^{2} \leq b^{2}$ for all $a,b\in \mathbb{R}$. Then, we have $[Q_{k}(t+1)]^{2} \leq [Q_{k}(t) + g_{k}^{t}(\mathbf{x}(t+1))]^{2}$. 

Summing over $k\in \{1,2,\ldots, m\}$ yields
\begin{align*}
\Vert \mathbf{Q}(t+1)\Vert^{2} \leq \Vert \mathbf{Q}(t) + \mathbf{g}^{t}(\mathbf{x}(t+1))\Vert^{2}.
\end{align*}
Thus, $\Vert \mathbf{Q}(t+1)\Vert \leq \Vert \mathbf{Q}(t) + \mathbf{g}^{t}(\mathbf{x}(t+1))\Vert \leq \Vert \mathbf{Q}(t)\Vert + \Vert \mathbf{g}^{t}(\mathbf{x}(t+1))\Vert \leq \Vert \mathbf{Q}(t)\Vert + G$ where the last inequality follows from Assumption \ref{as:basic}.

\item {\bf Proof of $\Vert \mathbf{Q}(t+1)\Vert \geq \Vert \mathbf{Q}(t)\Vert - G - \sqrt{m}D_{2} R$}:

Since $Q_k(t)\geq 0$, it follows that $|Q_k(t+1) - Q_k(t) | \leq |g_{k}^{t}(\mathbf{x}(t)) + [\nabla g_{k}^{t}(\mathbf{x}(t))]\tran [\mathbf{x}(t+1) - \mathbf{x}(t)] |$. (This can be shown by considering $g_{k}^{t}(\mathbf{x}(t)) + [\nabla g_{k}^{t}(\mathbf{x}(t))]\tran [\mathbf{x}(t+1) - \mathbf{x}(t)] \geq 0$ and $g_{k}^{t}(\mathbf{x}(t)) + [\nabla g_{k}^{t}(\mathbf{x}(t))]\tran [\mathbf{x}(t+1) - \mathbf{x}(t)] <0$ separately.)  Thus, we have $\Vert \mathbf{Q}(t+1) - \mathbf{Q}(t)\Vert \leq G+ \sqrt{m}D_2R$, which further implies $\Vert \mathbf{Q}(t+1)\Vert \geq \Vert \mathbf{Q}(t)\Vert - G - \sqrt{m}D_{2} R$ by the triangle inequality of norms.
\end{itemize}
\end{proof}
Now, we are ready to present the main proof of Lemma \ref{lm:queue-decrease-condition}. Note that Lemma \ref{lm:simple-queue-diff} gives $\big\vert \Vert \mathbf{Q}(t+1)\Vert - \Vert \mathbf{Q}(t)\Vert \big\vert \leq G+\sqrt{m} D_2 R$, which further implies that $\mathbb{E}[\Vert \mathbf{Q}(t+t_0)\Vert - \Vert \mathbf{Q}(t)\Vert |  \mathbf{Q}(t)] \leq t_0(G+\sqrt{m} D_2 R)$ when $\Vert \mathbf{Q}(t)\Vert < \theta$. It remains to prove $\mathbb{E}[\Vert \mathbf{Q}(t+1)\Vert - \Vert \mathbf{Q}(t)\Vert  \big|  \mathbf{Q}(t)]  \leq -\frac{\epsilon}{2}t_0$ when $\Vert \mathbf{Q}(t)\Vert \geq \theta$. Note that $\Vert \mathbf{Q}(0)\Vert = 0 <  \theta$.

Fix $t\geq 1$ and consider that $\Vert \mathbf{Q}(t)\Vert \geq \theta$. Let $\hat{\mathbf{x}}\in \mathcal{X}_{0}$ and $\epsilon>0$ be defined in Assumption \ref{as:interior-point}. Note that $\mathbb{E}[g^{t}_{k}(\hat{\mathbf{x}})] \leq -\epsilon, \forall k\in\{1,2,\ldots,m\}, \forall t\in\{1,2,\dots\}$ since $\omega(t)$ are i.i.d. from the distribution of $\omega$. Since $\hat{\mathbf{x}} \in \mathcal{X}_{0}$, by Lemma \ref{lm:decision-update-inequality}, for all $\tau\in \{t, t+1, \ldots,t+t_0-1\}$, we have
\begin{align*}
&V [\nabla f^{\tau}(\mathbf{x}(\tau))]\tran [\mathbf{x}(\tau+1) - \mathbf{x}(\tau)]+  \sum_{k=1}^{m} Q_{k}(\tau) [\nabla g^{\tau}_{k}(\mathbf{x}(\tau))]\tran [\mathbf{x}(\tau+1) - \mathbf{x}(\tau)] + \alpha \Vert \mathbf{x}(\tau+1) - \mathbf{x}(\tau)\Vert^{2} \\
\leq& V [\nabla f^{\tau}(\mathbf{x}(\tau))]\tran [\hat{\mathbf{x}} - \mathbf{x}(\tau)]+  \sum_{k=1}^{m} Q_{k}(\tau)[\nabla g^{\tau}_{k}(\mathbf{x}(\tau))]\tran [\mathbf{z} - \mathbf{x}(\tau)] + \alpha [\Vert \hat{\mathbf{x}}  - \mathbf{x}(\tau)\Vert^{2} - \Vert \hat{\mathbf{x}}  - \mathbf{x}(\tau+1)\Vert^{2}]. 
\end{align*}
Adding $\sum_{k=1}^{m} Q_k(\tau) g_k^{\tau}(\mathbf{x}(\tau))$ on both sides and noting that $g_k^{\tau}(\mathbf{x}(\tau)) + [\nabla g^{\tau}_{k}(\mathbf{x}(\tau))]\tran [\mathbf{z} - \mathbf{x}(\tau)] \leq g^{\tau}_k (\mathbf{z})$ by convexity yields
\begin{align*}
&V [\nabla f^{\tau}(\mathbf{x}(\tau))]\tran [\mathbf{x}(\tau+1) - \mathbf{x}(\tau)]+  \sum_{k=1}^{m} Q_{k}(\tau)\big[ g_k^{\tau}(\mathbf{x}(\tau)) + [\nabla g^{\tau}_{k}(\mathbf{x}(\tau))]\tran [\mathbf{x}(\tau+1) - \mathbf{x}(\tau)] \big]+ \alpha \Vert \mathbf{x}(\tau+1) - \mathbf{x}(\tau)\Vert^{2} \\
\leq& V [\nabla f^{\tau}(\mathbf{x}(\tau))]\tran [\hat{\mathbf{x}} - \mathbf{x}(\tau)]+  \sum_{k=1}^{m} Q_{k}(\tau) g_k^{\tau}(\mathbf{z}) + \alpha [\Vert \hat{\mathbf{x}}  - \mathbf{x}(\tau)\Vert^{2} - \Vert \hat{\mathbf{x}}  - \mathbf{x}(\tau+1)\Vert^{2}]. 
\end{align*}

Rearranging terms yields
\begin{align}
& \sum_{k=1}^{m} Q_{k}(t)\big[ g^{\tau}_{k}(\mathbf{x}(t)) + [\nabla g^{\tau}_{k}(\mathbf{x}(\tau))]\tran [\mathbf{x}(\tau+1) - \mathbf{x}(\tau)] \big] \nonumber\\
\leq &V [\nabla f^{\tau}(\mathbf{x}(\tau))]\tran [\hat{\mathbf{x}} - \mathbf{x}(\tau)] - V [\nabla f^{\tau}(\mathbf{x}(\tau))]\tran [\mathbf{x}(\tau+1) - \mathbf{x}(\tau)] + \alpha [\Vert \hat{\mathbf{x}}  - \mathbf{x}(\tau)\Vert^{2} - \Vert \hat{\mathbf{x}}  - \mathbf{x}(\tau+1)\Vert^{2}] \nonumber\\ & - \alpha \Vert \mathbf{x}(\tau+1) - \mathbf{x}(\tau)\Vert^{2} +  \sum_{k=1}^{m} Q_{k}(t)g_{k}^{\tau}(\hat{\mathbf{x}} )\nonumber\\
\leq & V [\nabla f^{\tau}(\mathbf{x}(\tau))]\tran [\hat{\mathbf{x}} - \mathbf{x}(\tau+1)]  +\alpha [\Vert \hat{\mathbf{x}}  - \mathbf{x}(\tau)\Vert^{2} - \Vert \hat{\mathbf{x}}  - \mathbf{x}(\tau+1)\Vert^{2}] +  \sum_{k=1}^{m} Q_{k}(\tau)g_{k}^{\tau}(\hat{\mathbf{x}} ) \nonumber\\
\overset{(a)}{\leq}&  V \Vert \nabla f^{\tau}(\mathbf{x}(\tau))\Vert \Vert \hat{\mathbf{x}} - \mathbf{x}(\tau+1)\Vert  +\alpha [\Vert \hat{\mathbf{x}}  - \mathbf{x}(\tau)\Vert^{2} - \Vert \hat{\mathbf{x}}  - \mathbf{x}(\tau+1)\Vert^{2}] +\sum_{k=1}^{m} Q_{k}(\tau)g_{k}^{\tau}(\hat{\mathbf{x}} ) \nonumber\\
\overset{(b)}{\leq} &VD_{1} R + \alpha [\Vert \hat{\mathbf{x}}  - \mathbf{x}(\tau)\Vert^{2} - \Vert \hat{\mathbf{x}}  - \mathbf{x}(\tau+1)\Vert^{2}]  +  \sum_{k=1}^{m} Q_{k}(\tau)g_{k}^{\tau}(\hat{\mathbf{x}} ), \label{eq:pf-queue-bound-eq1}
\end{align}
where (a) follows from the Cauchy-Schwarz inequality and (b) follows from Assumption \ref{as:basic}.

By Lemma \ref{lm:drift}, all $\tau\in \{t, t+1, \ldots, t+t_0-1\}$, we have
\begin{align*}
\Delta(\tau) \leq&  \sum_{k=1}^{m} Q_{k}(\tau)\big[ g_{k}^{\tau}(\mathbf{x}(\tau)) + [\nabla g^{\tau}_{k}(\mathbf{x}(\tau))]\tran [\mathbf{x}(\tau+1) - \mathbf{x}(\tau)]\big] +  \frac{1}{2}(G+\sqrt{m}D_{2}R)^{2}\\
\overset{(a)}{\leq}&  VD_{1} R +  \frac{1}{2}(G+\sqrt{m}D_{2}R)^{2} + \alpha [\Vert \hat{\mathbf{x}}  - \mathbf{x}(\tau)\Vert^{2} - \Vert \hat{\mathbf{x}}  - \mathbf{x}(\tau+1)\Vert^{2}]  +  \sum_{k=1}^{m} Q_{k}(\tau)g_{k}^{\tau}(\hat{\mathbf{x}} ),
\end{align*}
where (a) follows from \eqref{eq:pf-queue-bound-eq1}. 

Summing the above inequality over $\tau\in\{t, t+1, \ldots, t+t_0-1\}$, taking expectations conditional on $\mathcal{W}(t-1)$ on both sides and recalling that $\Delta(\tau) = \frac{1}{2} \Vert \mathbf{Q}(\tau+1)\Vert^{2} - \frac{1}{2} \Vert \mathbf{Q}(\tau)\Vert^{2}$ yields
\begin{align*}
&\mathbb{E}[ \Vert \mathbf{Q}(t+t_0)\Vert^{2} - \Vert \mathbf{Q}(t)\Vert^{2} \big|  \mathcal{W}(t-1)] \\
\leq & 2VD_{1} R t_0  +  t_0(G+\sqrt{m}D_{2}R)^{2} + 2\alpha \mathbb{E}[\Vert \hat{\mathbf{x}}  - \mathbf{x}(t)\Vert^{2} - \Vert \hat{\mathbf{x}}  - \mathbf{x}(t+t_0)\Vert^{2}| \mathcal{W}(t-1) ]  \\ &+ 2\sum_{\tau=t}^{t+t_0-1}\mathbb{E} [\sum_{k=1}^{m} Q_{k}(\tau)g_{k}^{\tau}(\hat{\mathbf{x}} ) | \mathcal{W}(t-1)] \\
\overset{(a)}{\leq} & 2VD_{1} R t_0  +  t_0(G+\sqrt{m}D_{2}R)^{2} + 2\alpha R^2 - 2\epsilon \sum_{\tau=t}^{t+t_0-1} \mathbb{E}[\Vert \mathbf{Q}(\tau) \Vert | \mathcal{W}(t-1)]\\ 
\overset{(b)}{\leq} &2VD_{1} R t_0  + t_0 (G+\sqrt{m}D_{2}R)^{2} + 2\alpha R^2 - 2\epsilon \sum_{\tau=0}^{t_0-1} \mathbb{E}[\Vert \mathbf{Q}(t) \Vert - \tau (G+\sqrt{m}D_2R)| \mathcal{W}(t-1)]\\
=& 2VD_{1} R t_0  +  t_0(G+\sqrt{m}D_{2}R)^{2} + 2\alpha R^2 - 2\epsilon t_0 \Vert \mathbf{Q}(t)\Vert + \epsilon t_0(t_0-1)(G+\sqrt{m}D_2R)\\
\leq& 2VD_{1} R t_0  +  t_0(G+\sqrt{m}D_{2}R)^{2} + 2\alpha R^2 - 2\epsilon t_0 \Vert \mathbf{Q}(t)\Vert + \epsilon t_0^2(G+\sqrt{m}D_2R)
\end{align*}
where (a) follows from $\Vert \hat{\mathbf{x}}  - \mathbf{x}(t)\Vert^{2} - \Vert \hat{\mathbf{x}}  - \mathbf{x}(t+t_0)\Vert^{2} \leq R^2$ by Assumption \ref{as:basic} and $\mathbb{E} [\sum_{k=1}^{m} Q_{k}(\tau)g_{k}^{\tau}(\hat{\mathbf{x}} ) | \mathcal{W}(t-1)] \leq -\epsilon  \mathbb{E}[\Vert \mathbf{Q}(\tau) \Vert | \mathcal{W}(t-1)], \forall \tau\in\{t, t+1, \ldots, t+t_0 -1\}$ by Lemma \ref{lm:Slater-negative}; and (b) follows from $\Vert \mathbf{Q}(t+1)\Vert \geq \Vert \mathbf{Q}(t)\Vert - (G+\sqrt{m}D_2R), \forall t$ by Lemma \ref{lm:simple-queue-diff}.

This inequality can be rewritten as  
\begin{align*}
&\mathbb{E}[ \Vert \mathbf{Q}(t+t_0)\Vert^{2} \big|  \mathcal{W}(t-1)] \\
\leq & \Vert \mathbf{Q}(t)\Vert^{2} - 2 \epsilon t_0\Vert \mathbf{Q}(t)\Vert + 2VD_{1} R t_0 + 2\alpha R^{2} +  t_0(G+\sqrt{m}D_{2}R)^{2} + \epsilon t_0^2 (G+\sqrt{m}D_2 R) \\
\overset{(a)}{\leq} & \Vert \mathbf{Q}(t)\Vert^{2} - \epsilon t_0\Vert \mathbf{Q}(t)\Vert  -\epsilon t_0[\frac{\epsilon}{2}t_0+ (G+\sqrt{m}D_2R)t_0 + \frac{2\alpha R^2}{t_0 \epsilon}+ \frac{2VD_{1} R +  (G+\sqrt{m}D_{2}R)^{2}}{\epsilon}] \\ &+ 2VD_{1} R t_0 + 2\alpha R^{2} + t_0(G+\sqrt{m}D_{2}R)^{2} + \epsilon t_0^2 (G+\sqrt{m}D_2 R)\\
=& \Vert \mathbf{Q}(t)\Vert^{2} - \epsilon t_0\Vert \mathbf{Q}(t)\Vert  - \frac{\epsilon^{2}t_0^2}{2}\\
\leq& [\Vert \mathbf{Q}(t)\Vert - \frac{\epsilon}{2}t_0]^{2},
\end{align*}
where (a) follows from the hypothesis that $\Vert \mathbf{Q}(t)\Vert \geq  \theta = \frac{\epsilon}{2}t_0+ (G+\sqrt{m}D_2R)t_0 + \frac{2\alpha R^2}{t_0 \epsilon}+ \frac{2VD_{1} R +  (G+\sqrt{m}D_{2}R)^{2}}{\epsilon}$.

Taking square root on both sides yields 
\begin{align*}
\sqrt{\mathbb{E}[ \Vert \mathbf{Q}(t+t_0)\Vert^{2} \big|  \mathcal{W}(t-1)]} \leq \Vert \mathbf{Q}(t)\Vert - \frac{\epsilon}{2}t_0.
\end{align*}

By the concavity of function $\sqrt{x}$ and Jensen's inequality, we have $$\mathbb{E}[ \Vert \mathbf{Q}(t+t_0)\Vert \big|  \mathcal{W}(t-1)] \leq \sqrt{\mathbb{E}[ \Vert \mathbf{Q}(t+t_0)\Vert^{2} | \mathcal{W}(t-1)]} \leq \Vert \mathbf{Q}(t)\Vert - \frac{\epsilon}{2}t_0.$$

\section{Proof of Lemma \ref{lm:extended-Azuma-inequality}} \label{app:pf-lm-extended-Azuma-inequality}

Intuitively, the second term on the right side in the lemma bounds the probability that $|Z(\tau+1)-Z(\tau)| >c$ for any $\tau\in\{0,1,\ldots, t\}$, while the first term on the right side comes from the conventional Hoeffding-Azuma inequality. However, it is unclear that whether $Z(t)$ is still a supermartigale conditional on the event that $|Z(\tau+1)-Z(\tau)| \leq c$ for any $\tau\in\{0,1,\ldots, t-1\}$.That's why it is important to have $\{ \vert Z(t+1) - Z(t) \vert > c\}\subseteq \{Y(t) > 0\}$ and $Y(t)\in \mathcal{F}(t)$, which means the boundedness of $\vert Z(t+1) - Z(t) \vert$ can be inferred from another random variable $Y(t)$ that belongs to $\mathcal{F}(t)$. The proof of Lemma \ref{lm:extended-Azuma-inequality} uses the truncation method to construct an auxiliary supermargingale.

Recall the definition of stoping time given as follows:
\begin{Def}[\cite{book_DurrettProbabilityTE}]
Let $\{\emptyset, \Omega\}=\mathcal{F}(0)\subseteq\mathcal{F}(1)\subseteq\mathcal{F}(2)\cdots$ be a filtration. A discrete random variable $T$ is a stoping time (also known as an option time) if for any integer $t<\infty$,
\[\{T=t\}\in\mathcal{F}(t),\]
i.e. the event that the stopping time occurs at time $t$ is contained in the information up to time $t$.
\end{Def}

The next theorem summarizes  that a supermartingale truncated at a stoping time is still a supermartingale.
 
\begin{Thm}\label{thm:stopping-time-martingale}
(\textit{Theorem 5.2.6 in \cite{book_DurrettProbabilityTE}}) If random variable $T$ is a stopping time and $Z(t)$ is a supermartingale, then $Z(t\wedge T)$ is also a supermartingale, where $a\wedge b\triangleq\min\{a,b\}$.
\end{Thm}

To prove this lemma, we first construct a new supermartingale by truncating the original supermartingale at a carefully chosen stopping time such that the new supermartingale has bounded differences.

Define integer random variable $T = \inf\{t\geq 0: Y(t) > 0 \} $.  That is, $T$ is the first time $t$ when $Y(t) > 0$ happens. Now, we show that $T$ is a stoping time and if we define $\widetilde{Z}(t) = Z(t\wedge T)$, then $\{\widetilde{Z}(t) \neq Z(t)\} \subseteq \bigcup_{\tau=0}^{t-1} \{ Y(\tau) > 0\}, \forall t\geq 1$ and $\widetilde{Z}(t)$ is a supermartingale with differences bounded by $c$ . 
\begin{enumerate}[leftmargin=20pt]
\item {\bf To show $T$ is a stoping time:}  Note that $\{T = 0\} = \{Y(0)>0\}\in \mathcal{F}(0)$. Fix integer $t^{\prime} >0$, we have
\begin{align*}
\{T = t^{\prime}\} =& \big\{ \inf \{t\geq0: Y(t)> 0\} = t^{\prime}\big\} \\
=& \big\{\cap_{\tau=0}^{t^{\prime}-1}\{ \vert Y(\tau)\leq 0\} \big\}\cap \{Y(t^{\prime}) > 0\} \\
\overset{(a)}{\in} & \mathcal{F}(t^{\prime})
\end{align*}
where (a) follows because $\{ Y(\tau) \leq 0\}  \in \mathcal{F}(\tau)  \subseteq \mathcal{F}(t^{\prime})$ for all $\tau\in\{0,1,\ldots,t^{\prime}-1\}$ and $\{ Y(t^{\prime}) > 0\} \in \mathcal{F}(t^{\prime})$. It follows that $T$ is a stoping time.
\item {\bf To show $\{\widetilde{Z}(t) \neq Z(t)\} \subseteq \bigcup_{\tau=0}^{t-1} \{Y(\tau) > 0\}, \forall t\geq 1$:}  Fix $t = t^{\prime}> 1$. Note that
\begin{align*}
\{\widetilde{Z}(t^{\prime}) \neq Z(t^{\prime})\} \overset{(a)}{\subseteq}& \{ T <t^{\prime}\} = \big\{\inf \{t>0:  Y(t) >0\} < t^{\prime}\big\}\\
\subseteq& \bigcup_{\tau=0}^{t^{\prime}-1} \{ Y(\tau)  >0\}
\end{align*}
where (a) follows by noting that if $T\geq t^{\prime}$ then $\widetilde{Z}(t^{\prime}) = Z(t^{\prime}\wedge T) = Z(t^{\prime})$.
\item {\bf To show $\widetilde{Z}(t)$ is a supermartingale with differences bounded by $c$:}  Since random variable $T$ is proven to be a stoping time,  $\widetilde{Z}(t) = Z(t\wedge T)$ is a supermartingale by Theorem \ref{thm:stopping-time-martingale}. It remains to show  $\vert \widetilde{Z}(t+1) - \widetilde{Z}(t)\vert \leq c, \forall t\geq 0$. Fix integer $t = t^{\prime}\geq0$. Note that 
\begin{align*}
\vert   \widetilde{Z}(t^{\prime}+1) - \widetilde{Z}(t^{\prime}) \vert   =& \vert   Z(T\wedge (t^{\prime}+1)) -Z(T\wedge t^{\prime}) \vert  \\
=& \vert   \mathbf{1}_{\{T\geq t^{\prime}+1\}}[Z(T\wedge (t^{\prime}+1)) -Z(T\wedge t^{\prime})] + \mathbf{1}_{\{T\leq t^{\prime}\}}[Z(T\wedge (t^{\prime}+1)) -Z(T\wedge t^{\prime})] \vert   \\
=&  \vert   \mathbf{1}_{\{T\geq t^{\prime}+1\}}[Z(t^{\prime}+1) -Z(t^{\prime})] + \mathbf{1}_{\{T\leq t^{\prime}\}}[Z(T) -Z(T)] \vert   \\
=& \mathbf{1}_{\{T\geq t^{\prime}+1\}}\vert  Z(t^{\prime}+1) -Z(t^{\prime})\vert
\end{align*}
Now consider $T\leq t^{\prime}$ and $T\geq t^{\prime}+1$ separately. 
\begin{itemize}
\item In the case when $T \leq t^{\prime}$, it is straightforward that $\vert \widetilde{Z}(t^{\prime}+1) - \widetilde{Z}(t^{\prime}) \vert= \mathbf{1}_{\{T\geq t^{\prime}+1\}}\vert  Z(t^{\prime}+1) -Z(t^{\prime})\vert =0 \leq c$. 
\item Consider the case when $T\geq t^{\prime}+1$. By the definition of $T$, we know that $\{T\geq t^{\prime}+1\} =\big\{ \inf \{t\geq0: Y(t)> 0\} \geq t^{\prime} +1\big\} \subseteq \bigcap_{\tau=0}^{t^{\prime}} \{ Y(\tau) \leq 0\} \subseteq \bigcap_{\tau=0}^{t^{\prime}} \{ \vert Z(\tau+1) - Z(\tau) \vert \leq c\}$, where the last inclusion follows from the fact that $\{\vert Z(\tau+1) - Z(\tau) \vert > c\} \subseteq \{Y(\tau) >0\}$. That is, when $T\geq t^{\prime}+1$, we must have $|Z(\tau+1) - Z(\tau)| \leq c$ for all $\tau\in\{1,\ldots, t^{\prime}\}$, which further implies that $\vert Z(t^{\prime}+1) -Z(t^{\prime})\vert\leq c$. Thus,  when $T\geq t^{\prime}+1$, $\vert \widetilde{Z}(t^{\prime}+1) - \widetilde{Z}(t^{\prime}) \vert = \mathbf{1}_{\{T\geq t^{\prime}+1\}}\vert Z(t^{\prime}+1) -Z(t^{\prime})\vert \leq c$. 
\end{itemize}
Combining two cases together proves $\vert \widetilde{Z}(t^{\prime}+1) - \widetilde{Z}(t^{\prime}) \vert \leq c$. 
\end{enumerate}

Since $\widetilde{Z}(t)$ is a supermartingale with bounded differences $c$ and $\widetilde{Z}(0) = Z(0) = 0$, by the conventional Hoeffding-Azuma inequality, for any $z>0$, we have
\begin{align}
\text{Pr}(\widetilde{Z}(t)\geq z) \leq e^{-z^2/(2tc^2)} \label{eq:Z(t)-azuma}
\end{align}

Finally, we have
\begin{align*}
\text{Pr}(Z(t) \geq z) = & \text{Pr}(\widetilde{Z}(t) = Z(t), Z(t) \geq z) + \text{Pr}(\widetilde{Z}(t)\neq Z(t), Z(t) \geq z)\\
\leq & \text{Pr}(\widetilde{Z}(t) \geq z) +   \text{Pr}(\widetilde{Z}(t)\neq Z(t))\\
\overset{(a)}{\leq}& e^{-z^2/(2tc^2)} + \text{Pr}(\bigcup_{\tau=0}^{t-1} Y(\tau) > 0)\\
\overset{(b)}{\leq}&e^{-z^2/(2tc^2)} + \sum_{\tau=0}^{t-1} p(\tau)
\end{align*}
where (a) follows from equation \eqref{eq:Z(t)-azuma} and the second bullet in the above; and (b) follows from the union bound and the hypothesis that $\text{Pr}(Y(\tau)>0 )\leq p(\tau), \forall \tau$.

\section{Proof of Theorem \ref{thm:high-prob-regret-bound}} \label{app:thm-high-prob-regret-bound}

Define $Z(t) = \sum_{\tau=1}^{t} \sum_{k=1}^m Q_k(\tau) g^{\tau}_k(\mathbf{x^\ast})$. Recall $\mathcal{W}(t) = \sigma(\omega(1), \ldots, \omega(t))$. The next lemma shows that  $Z(t)$ satisfies Lemma \ref{lm:extended-Azuma-inequality} with $\mathcal{F}(t) = \mathcal{W}(t)$ and $Y(t) = \Vert \mathbf{Q}(t+1)\Vert - \frac{c}{G}$.
 
\begin{Lem}\label{lm:Qg-sum-concentration}
Let $\mathbf{x}^{\ast}\in \mathcal{X}_{0}$ be any fixed solution that satisfies $\tilde{\mathbf{g}}(\mathbf{x}^{\ast}) \leq \mathbf{0}$, e.g., $\mathbf{x}^{\ast} = \argmin_{\mathbf{x}\in \mathcal{X}} \sum_{t=1}^{T} f^{t}(\mathbf{x})$. Under Algorithm \ref{alg:new-alg}, if we define $Z(0)=0$ and $Z(t) = \sum_{\tau=1}^{t} \sum_{k=1}^m Q_k(\tau) g^{\tau}_k(\mathbf{x^\ast}), \forall t\geq 1$, then $\{Z(t), t\geq 0\}$ is a supermartingale adapted to filtration $\{\mathcal{W}(t), t\geq 0\}$ such that $\{| Z(t+1) - Z(t)| > c\} \subseteq \{Y(t) >0\}, \forall t\geq 0$ where $Y(t) = \Vert \mathbf{Q}(t+1)\Vert - \frac{c}{G}$ is a random variable adapted to  $\mathcal{W}(t)$.
\end{Lem}

\begin{proof}
It is easy to say $\{Z(t), t\geq 0\}$ is adapted $\{\mathcal{W}(t), t\geq0\}$. It remains to show $\{Z(t), t\geq 0\}$ is a supermartingale. Note that $Z(t+1) = Z(t) + \sum_{k=1}^m Q_k(t+1) g^{t+1}_k(\mathbf{x^\ast})$ and 
\begin{align*}
\mathbb{E}[Z(t+1)|\mathcal{W}(t)] =& \mathbb{E}[Z(t) + \sum_{k=1}^m Q_k(t+1) g^{t+1}_k(\mathbf{x^\ast}) | \mathcal{W}(t) ]\\
\overset{(a)}{=}&Z(t) + \sum_{k=1}^m Q_k(t+1)\mathbb{E}[g^{t+1}_k(\mathbf{x^\ast})] \\
\overset{(b)}{\leq} &Z(t)
\end{align*}
where (a) follows from the fact that $Z(t)\in \mathcal{W}(t)$, $\mathbf{Q}(t+1) \in \mathcal{W}(t)$ and $\mathbf{g}^{t+1}(\mathbf{x}^\ast)$ is independent of  $\mathcal{W}(t)$; and (b) follows from $\mathbb{E}[g^{t+1}_k(\mathbf{x^\ast})] = \tilde{g}_k(\mathbf{x}^{\ast}) \leq 0$ which further follows from $\omega(t)$ are i.i.d. samples. Thus,  $\{Z(t), t\geq 0\}$ is a supermartingale.

We further note that 
\begin{align*}
|Z(t+1) - Z(t)| = | \sum_{k=1}^m Q_k(t+1) g^{t+1}_k(\mathbf{x^\ast})| \overset{(a)}{\leq} \Vert \mathbf{Q}(t+1)\Vert G
\end{align*}
where (a) follows from the Cauchy-Schwarz inequality and  the assumption that $\Vert \mathbf{g}^t (\mathbf{x}^\ast)\Vert\leq G$. 

This implies that if $|Z(t+1) - Z(t)| > c$, then $\Vert \mathbf{Q}(t)\Vert > \frac{c}{G}$. Thus, $\{|Z(t+1) - Z(t)| > c\} \subseteq \{\Vert \mathbf{Q}(t+1)\Vert > \frac{c}{G}\}$. Since $\mathbf{Q}(t+1)$ is adapted to $\mathcal{W}(t)$, it follows that $Y(t) = \Vert \mathbf{Q}(t+1)\Vert - \frac{c}{G}$ is a random variable adapted to  $\mathcal{W}(t)$.
\end{proof}

By  Lemma \ref{lm:Qg-sum-concentration}, $Z(t)$ satisfies Lemma \ref{lm:extended-Azuma-inequality}. Fix $T\geq 1$, Lemma  \ref{lm:extended-Azuma-inequality} implies that   
\begin{align}
\text{Pr}(\sum_{t=1}^{T} \sum_{k=1}^m Q_k(t) g^{t}_k(\mathbf{x^\ast}) \geq \gamma) \leq \underbrace{e^{-\gamma^2/(2Tc^2)}}_{\text{(I)}} + \underbrace{\sum_{t=0}^{T-1} \text{Pr}(\Vert \mathbf{Q}(t+1) \Vert> \frac{c}{G})}_{\text{(II)}} \label{eq:pf-high-prob-constraint-eq1}
\end{align}

Fix $0<\lambda<1$. In the following, we shall choose $\gamma$ and $c$ such that both term (I) and term (II) in \eqref{eq:pf-high-prob-constraint-eq1} are no larger than $\frac{\lambda}{2}$.

Recall that by Lemma \ref{lm:queue-decrease-condition}, random process $\widetilde{Z}(t) = \Vert \mathbf{Q}(t)\Vert$ satisfies the conditions in Lemma \ref{lm:drift-random-process-bound}. To guarantee term (II) is no lareger than $\frac{\lambda}{2}$, it suffices to choose $c$ such that 
\begin{align*}
\text{Pr}(\Vert \mathbf{Q}(t)\Vert > \frac{c}{G}) < \frac{\lambda}{2T}, \forall t\in\{1,2,\ldots, T\}
\end{align*}
By part (2) of Lemma \ref{lm:drift-random-process-bound} (with $\mu = \frac{\lambda}{2T}$), the above inequality holds if we choose $c = t_0\frac{\epsilon}{2}G+ t_0 (G+\sqrt{m}D_2R)G+ \frac{2\alpha R^2}{t_0 \epsilon}G+ \frac{2VD_{1} R +  (G+\sqrt{m}D_{2}R)^{2}}{\epsilon}G+ t_0BG + t_0 \frac{8(G+\sqrt{m}D_{2}R)^{2}}{\epsilon} \log(\frac{2T}{\lambda})G$  where $$B = \frac{8(G+\sqrt{m}D_{2}R)^{2}}{\epsilon}\log[1+ \frac{32(G+\sqrt{m}D_{2}R)^{2}}{\epsilon^{2}} e^{\epsilon/[8(G+\sqrt{m}D_{2}R)]}]$$ is an absolute constant irrelevant the algorithm parameters and $t_0>0$ is an arbitrary integer.

Once $c$ is chosen, we further need to choose $\gamma$ such that term (I) in \eqref{eq:pf-high-prob-constraint-eq1} is $\frac{\lambda}{2}$. It follows that if $\gamma = \sqrt{T} \log^{0.5}(\frac{1}{\lambda}) c = \sqrt{T} \log^{0.5}(\frac{1}{\lambda}) [\frac{\epsilon}{2}t_0 G+ t_0(G+\sqrt{m}D_2R)G + \frac{2\alpha R^2}{t_0 \epsilon}G+ \frac{2VD_{1} R +  (G+\sqrt{m}D_{2}R)^{2}}{\epsilon} G+ t_0BG + t_0 \frac{8(G+\sqrt{m}D_{2}R)^{2}}{\epsilon}\log(\frac{2T}{\lambda})G]$, then 
\begin{align*}
\text{Pr}(\sum_{t=1}^{T} \sum_{k=1}^m Q_k(t) g^{t}_k(\mathbf{x^\ast}) \geq \gamma) \leq \lambda, 
\end{align*} 
or equivalently, 
\begin{align}
\text{Pr}(\sum_{t=1}^{T} \sum_{k=1}^m Q_k(t) g^{t}_k(\mathbf{x^\ast}) < \gamma) \geq 1-\lambda. \label{eq:pf-high-prob-constraint-eq2}
\end{align}

Note that if we take $t_0 = \lceil \sqrt{T}\rceil$, $V =\sqrt{T}$ and $\alpha = T$, then $\gamma = O\big(T\log(T)\log^{0.5}(\frac{1}{\lambda})\big) +O\big( T\log^{1.5}(\frac{1}{\lambda})\big) = O\big(T\log(T)\log^{1.5}(\frac{1}{\lambda})\big)$.

 By Lemma \ref{lm:deterministic-regret-bound}  (with $\mathbf{z} = \mathbf{x}^\ast$, $V =\sqrt{T}$ and $\alpha = T$), we have 
\begin{align}
\sum_{t=1}^{T}f^{t}(\mathbf{x}(t))
\leq &\sum_{t=1}^{T}f^{t}(\mathbf{x}^{\ast}) + \sqrt{T} R^{2} + \frac{D_{1}^{2}}{4} \sqrt{T} + \frac{1}{2}(G+\sqrt{m}D_{2}R)^{2} \sqrt{T} + \frac{1}{\sqrt{T}}\sum_{t=1}^{T} \big[\sum_{k=1}^{m} Q_{k}(t) g_{k}^{t}(\mathbf{x}^{\ast})\big] \label{eq:pf-high-prob-constraint-eq3}
\end{align}
Substituting \eqref{eq:pf-high-prob-constraint-eq2} into \eqref{eq:pf-high-prob-constraint-eq3} yields
\begin{align*}
\text{Pr}\Big( \sum_{t=1}^{T}f^{t}(\mathbf{x}(t)) \leq \sum_{t=1}^{T}f^{t}(\mathbf{x}^{\ast})  + O\big(T\log(T)\log^{1.5}(\frac{1}{\lambda})\big) \Big) \geq 1-\lambda.
\end{align*}

\section{More Experiment Details} \label{app:experiment}
In the experiment, we assume the job arrivals $\omega(t)$ are Poisson distributed with mean $1000$ jobs/slot. For simplicity, assume each server is restricted to choose power $x_{i}(t)\in[0,30]$ at each round and the service rate satisfies $h_i(x_i(t)) = 4\log(1+4x_i(t))$. (Note that our algorithm can easily deal with general concave functions $h_i(\cdot)$ and each server in general can have different $h_i(\cdot)$ functions.) The simulation duration is $2160$ slots (corresponding to $10$ days).
 
 The three baselines are further elaborated as below:
\begin{itemize}[leftmargin=20pt]
 \item Best fixed decision in hindsight: Assume all the electricity price traces and the job arrival distribution are known beforehand. The decision maker chooses a  fixed power decision vector $\mathbf{p}^\ast$ that is optimal based on data in $2160$ slots.
 \item React algorithm: This algorithm is developed in \cite{gandhi2012sleep}. The algorithm reacts to the current traffic and splits the load evenly among each server to support the arrivals.  Since instantaneous job arrivals is unknown at the current slot, we use the average of job arrivals over the most recent $5$ slots as an estimate.  Since this algorithm is designed to meet the time varying job arrivals but is unaware of electricity variations, its electricity cost is high as observed in our simulation results.
 \item Low-power algorithm: This algorithm is adapted from \cite{qureshi2009cutting} and always schedule jobs to servers in the zones with the lowest electricity price. Since instantaneous electricity prices are unknown at the current slot, we use the average of electricity prices over the most recent $5$ slots at each server as an estimate. Recall that each server has a finite service capacity ($x_i(t)\in[0,30]$), this algorithm is not guaranteed serve all job arrivals. Thus, the number of unserved jobs can eventually pile up.
\end{itemize}

\bibliographystyle{IEEEtran}
%\bibliographystyle{plain}
%% argument is your BibTeX string definitions and bibliography database(s)
\bibliography{mybibfile}

\end{document}